\documentclass{amsart}
\usepackage{amsmath,amssymb,amsfonts}
\usepackage{mathrsfs,latexsym,amsthm,enumerate}
\usepackage{tikz}
\usepackage{amscd}
\usepackage[all]{xy}

\newtheorem{lemma}{Lemma}[section]
\newtheorem{proposition}[lemma]{Proposition}
\newtheorem{corollary}[lemma]{Corollary}

\newtheorem{remark}[lemma]{Remark}
\newtheorem{remarks}[lemma]{Remarks}
\newtheorem{theorem}[lemma]{Theorem}
\newtheorem{example}[lemma]{Example}
\newtheorem{examples}[lemma]{Examples}
\newtheorem{conjecture}[lemma]{Conjecture}

\begin{document}

\title[Boolean inverse monoids]{On a class of countable Boolean inverse monoids and Matui's spatial realization theorem}

\author{Mark V. Lawson}

\address{Department of Mathematics
and the
Maxwell Institute for Mathematical Sciences,
Heriot-Watt University,
Riccarton,
Edinburgh~EH14~4AS,
United Kingdom}  
\email{markl@hw.ac.uk }
\thanks{The author was partially supported by an EPSRC grant  (EP/I033203/1). }

\begin{abstract} We introduce a class of inverse monoids that can be regarded as non-commutative generalizations of Boolean algebras.
These inverse monoids are related to a class of \'etale topological groupoids,
under a non-commutative generalization of classical Stone duality.
Furthermore, and significantly for this paper, they arise naturally in the theory of dynamical systems as developed by Matui.
We are thereby able to reinterpret a theorem of Matui on a class of \'etale groupoids, in the spirit of Rubin's theorem, 
as an equivalent theorem about a class of inverse monoids.
The inverse monoids in question may be viewed as the countably infinite generalizations of finite symmetric inverse monoids.
Their groups of units therefore generalize the finite symmetric groups and include amongst their number the Thompson groups $G_{n,1}$.
 \end{abstract}



\maketitle

\section{Introduction}

The argument for studying inverse semigroups is that they algebraically encode information about partial symmetries 
generalizing the way that groups algebraically encode information about symmetries.
Thus the symmetric inverse monoids $I(X)$, the monoids of all partial bijections on the set $X$, generalize the symmetric groups
which in fact arise as their groups of units.
The Wagner-Preston representation theorem, the Cayley theorem for these structures, 
says that every inverse semigroup is isomorphic to an inverse semigroup of partial bijections.
Inverse semigroups were introduced as the correct algebraic setting for studying pseudogroups of transformations,
where a pseudogroup is simply an inverse semigroup of all partial homeomorphisms between the open subsets of a topological space.
The basic properties of inverse semigroups, and further historical and motivational issues, are discussed in \cite{Lawson98}.
Compelling evidence for the role of inverse semigroups as carriers of information about partial symmetries was provided by the the work of Kellendonk \cite{Kell1,Kell2}.
He showed that  the inverse semigroups of partial translational symmetries of aperiodic tilings were an important ingredient in calculating invariants
of those tilings that had physical interpretations.

Kellendonk's work was also a further example of the role played by inverse semigroup theory in studying $C^{\ast}$-algebras.
The connection between inverse semigroups and $C^{\ast}$-algebras originated in Renault's monograph \cite{Renault}
and has developed into an important area of application.
The papers \cite{Kumjian,Exel1,Exel2,Lenz} and the monograph \cite{Paterson} are just a representative sample.
$C^{\ast}$-algebras are suitable vessels for representing inverse semigroups but the representations that arise are actually mediated by a third class of structures: topological groupoids.

The fact that {\em pseudogroups} and topological groupoids are related has become virtually folklore.
The groupoid of germs of a pseudogroup is a standard construction and it is this  groupoid that has usually been the structure studied rather than the pseudogroup itself.
On the other hand, from a topological groupoid one may construct a pseudogroup by taking local sections, usually with extra structure of some kind.
What is curious is that the nature of this connection was not explored further until comparatively recently.
The catalyst was Paterson's monograph \cite{Paterson}.
There were initially two different approaches.
The first started with the theory of quantales and localic groupoids \cite{Res1,Res2} and could be viewed as a non-commutative frame theory \cite{PJ}.
The second started as a non-commutative version of Stone duality \cite{LMS,Lawson10a,Lawson10b,Lawson12,LL,LL2} developing some important insights by Lenz \cite{Lenz}
who in turn was attempting to understand Kellendonk's basic constructions from an algebraic point of view.
Recent work has shown that these two approaches are complementary \cite{KL}.
As a consequence, we now understand the nature of the connections between inverse semigroups, and therefore pseudogroups, and topological groupoids.
In this paper, I will need just one case of this general theory of non-commutative Stone duality.

Every inverse semigroup comes equipped with an algebraically defined natural partial order.
This is the abstract version of the restriction order on partial bijections.
Although this order has always played an important part in the theory, comparatively little was done to study
inverse semigroups with what might be called lattice-like properties.
A notable exception was Leech's work \cite{Leech} on inverse semigroups in which every pair of elements has a meet,
the so-called {\em inverse $\wedge$-semigroups}.
His work hinted at a connection between such inverse semigroups and groups since 
the $\wedge$-completion of a group yields an inverse semigroup of cosets.
However, non-commutative Stone duality requires further strong order-theoretic properties
which we now describe.

The {\em compatibility relation}, denoted by $\sim$, is defined by
$a \sim b$ if, and only if, $ab^{-1}$ and $a^{-1}b$ are both idempotents.
If $a,b \leq c$ then $a \sim b$.
Thus being compatible is a necessary condition for a pair of elements to have a join.
An inverse semigroup is said to be {\em distributive} if it has all binary joins of compatible pairs of elements and
multiplication distributes over any binary joins that exist.
The set of idempotents $E(S)$ of an inverse semigroup $S$ forms an idempotent commutative subsemigroup
and so with respect to the natural partial order a meet-semilattice, usually referred to as the {\em semilattice of idempotents}.
An inverse monoid is said to be a {\em Boolean inverse monoid} if it
is a distributive inverse monoid whose semilattice of idempotents is a (unital) Boolean algebra.
We shall be interested in {\em Boolean inverse $\wedge$-monoids},
prime examples of which are the symmetric inverse monoids.
Under non-commutative Stone duality, such inverse monoids correspond
to Hausdorff \'etale topological groupoids whose space of identities is a Stone space,
that is, a compact Hausdorff space with a basis of clopen subsets.

We now come to the connection that motivates this whole paper.
Groupoids of this complection arise naturally in mathematics quite independently of inverse semigroup theory.
Most recently, and importantly for this paper, in work by Matui \cite{Matui12,Matui13} who is motivated by questions coming from topological dynamics.
Non-commutative Stone duality tells us that his results should be interpreted as results about a class of 
countable Boolean inverse $\wedge$-monoids.
These, then, will be the subject of this paper.

To state the main theorem to be proved, we need  two definitions.
Let $S$ be a Boolean inverse monoid.
An ideal $I$ of $S$ is said to be {\em $\vee$-closed} if
$a \vee b \in I$ whenever $a,b \in I$ and $a$ and $b$ are compatible.
A $\vee$-closed ideal will be called a {\em $\vee$-ideal}.
If the only $\vee$-ideals of $S$ are $\{0\}$ and $S$ we say that $S$ is {\em $0$-simplifying}.
We denote by $Z(E(S))$ the centralizer of the idempotents in $S$.
Clearly, $E(S) \subseteq Z(E(S))$.
If $E(S) = Z(E(S))$ then $S$ is said to be {\em fundamental}.
Symmetric inverse monoids are $0$-simplifying and fundamental and, in the finite case,
the only Boolean inverse $\wedge$-monoids satisfying these two conditions.

\begin{remark}{\em It is a theorem of Wagner \cite[Theorem~5.2.10]{Lawson98} that fundamental inverse semigroups are precisely those isomorphic
to inverse semigroups of partial homeomorphisms of a $T_{0}$-space where the domains of definition form a basis for the topology.
See \cite{Lawson98}.}
\end{remark}

We can now state our main theorem  which is the inverse monoid version of \cite[Theorem 3.10]{Matui13}.

\begin{theorem}[Spatial Realization Theorem]\label{them:one} 
Let $S$ and $T$ be two countably infinite Boolean inverse $\wedge$-monoids
which are $0$-simplifying and fundamental.
Then the following are equivalent.
\begin{enumerate}

\item $S$ and $T$ are isomorphic.

\item The groups of units of $S$ and $T$ are isomorphic.

\end{enumerate}
\end{theorem}

This result is seemingly paradoxical from the point of view of the theory of partial symmetries
since  the global symmetries determine the partial symmetries, which is decidedly puzzling.
Examples of inverse monoids satisfying the conditions of the above theorem
are the {\em Cuntz inverse monoids $C_{n}$} described in \cite{Lawson07a,Lawson07b}.
The groups of units of such monoids are the Thompson groups $G_{n,1}$.
This shows the highly non-trivial nature of the theorem.
The Cuntz inverse monoids are congruence-free and so the following corollary is non-trivial.

\begin{corollary}\label{cor:congruence-free}
Let $S$ and $T$ be two countably infinite, congruence-free Boolean inverse $\wedge$-monoids.
Then the following are equivalent.
\begin{enumerate}

\item $S$ and $T$ are isomorphic.

\item The groups of units of $S$ and $T$ are isomorphic.

\end{enumerate}
\end{corollary}

\section{Elementary background}

Let $S$ be an inverse semigroup.
If $A \subseteq S$ then $E(A) = A \cap E(S)$.
If $s \in S$ then $s^{-1}s$ is called the {\em domain idempotent} and $ss^{-1}$ is called the {\em range idempotent}.
We shall sometimes write $\mathbf{d}(s) = s^{-1}s$ and $\mathbf{r}(s) = ss^{-1}$.
We write $s^{-1}s \, \mathscr{D} \, ss^{-1}$ and sometimes use the notation
$s^{-1}s \stackrel{s}{\rightarrow} ss^{-1}$.
If $s,t \in S$ are such that $\mathbf{d}(s) = \mathbf{r}(t)$ then we say that $st$ is a {\em restricted product}.
Observe that in this case $\mathbf{d}(st) = \mathbf{d}(t)$ and $\mathbf{r}(st) = \mathbf{r}(s)$.
Our inverse semigroups will always be assumed to have a zero and, in addition, this paper deals with monoids.
A {\em unit} is simply an invertible element.
A unit $g$ is {\em non-trivial} if $g \neq 1$.
The group of units of $S$ is denoted by $U(S)$.
If $g,h \in U(S)$, define $[g,h] = ghg^{-1}h^{-1}$, the {\em commutator} of $g$ and $h$.
An {\em involution} is a unit $g$ such that $g^{2} = 1$.
A non-zero element $s \in S$ is called an {\em infinitesimal} if $s^{2} = 0$.
Infinitesimals will play an important role in constructing units.
One of the themes of this paper will be the relationship between an inverse monoid and its group of units.
If $e$ is any idempotent in $S$ then $eSe$ is an inverse subsemigroup that is a monoid with respect to $e$ called a {\em local submonoid}.
An element of an inverse semigroup is said to be an {\em atom} if it is non-zero and the only element strictly below it is zero.
In this paper, all Boolean algebras will be {\em unital}.
If $e$ is an element of a Boolean algebra then $\bar{e}$ denotes its complement.
A Boolean algebra is called {\em atomless} if it has no atoms.
It is a famous theorem of Tarski that any two countable atomless Boolean algebras are isomorphic \cite{GH}.
Accordingly, we call the unique countable atomless Boolean algebra the {\em Tarski algebra}.
The Stone space of the Tarski algebra is the {\em Cantor space}.
If  $ab^{-1}$ and $a^{-1}b$ are both zero we say that $a$ and $b$ are {\em orthogonal}.
In this case, we often write $a \perp b$.
If the join of an orthogonal pair of elements exists, we shall say that it is an {\em orthogonal join}.

We now recall the standard definition of the {\em maximum idempotent-separating congruence} $\mu$ \cite{Lawson98}.
Let $S$ be an arbitrary inverse semigroup.
Define $s  \, \mu \, t$ if and only if $ses^{-1} = tet^{-1}$ for all idempotents $e \in E(S)$.
Thus $s$ and $t$ induce the same conjugation maps on $E(S)$.
Observe that if $s \, \mu \, t$ then $s^{-1}s = t^{-1}t$ and $ss^{-1} = tt^{-1}$.
In fact, if $s^{-1}s = t^{-1}t$ and $ss^{-1} = tt^{-1}$ then to check that  $s  \, \mu \, t$ it is enough to verify  $ses^{-1} = tet^{-1}$ for all idempotents $e \leq s^{-1}s$.
An inverse semigroup is fundamental if, and only if, $\mu$ is equality \cite[Proposition~5.2.5]{Lawson98}.
For each sermilattice $E$ we may define an inverse semigroup $T_{E}$ called the {\em Munn semigroup} of $E$.
This consists of all order-isomorphisms between the principal order ideals of $E$.
Its semilattice of idempotents is isomorphic to $E$.
Given an inverse semigroup $S$ there is a homomorphism to $T_{E(S)}$ that contains all the idempotents of $T_{E(S)}$.
This homomorphism is injective if, and only if, the semigroup $S$ is fundamental.
A non-trivial inverse semigroup with zero is said to be {\em $0$-simple} if the only ideals are the zero ideal and the whole inverse semigroup.
The following is an alternative characterization which is more useful.
The proof can be found in \cite{Lawson98}.

\begin{lemma}\label{lem:helen}
Let $S$ be an inverse semigroup with zero.
It is $0$-simple if, and only if, for any two non-zero idempotents $e$ and $f$
there exists an idempotent $i$ such that $e \, \mathscr{D} \, i \leq f$. 
\end{lemma}

A meet semilattice $E$ with zero is said to be {\em $0$-disjunctive} if for all non-zero $e \in E$ and $0 \neq f < e$ there exists
$0 \neq f' \leq e$ such that $f \wedge f' = 0$.
Observe that Boolean algebras are automatically $0$-disjunctive.
An inverse semigroup is said to be {\em congruence-free} if it has exactly two congruences.
It is a standard theorem that an inverse semigroup with zero is congruence-free if and only if
it is fundamental, $0$-simple and its semilattice of idempotents is $0$-disjunctive \cite{Petrich}.

It is clear that 0-simple Boolean inverse monoids are 0-simplifying since they have no non-trivial ideals at all.
However, the converse is not true.
The following was first discussed as Example~4.14 of \cite{Lawson12}.

\begin{example} {\em Let $I(X)$ be the finite symmetric inverse monoid on the set $X = \{1,2,\ldots, n\}$.
We shall usually deote it by $I_{n}$.
This is a Boolean inverse $\wedge$-monoid.
For $n \geq 2$, this monoid is not 0-simple in that there are non-trivial ideals.
We prove that it is, nevertheless, 0-simplifying.
Let $I \subseteq I(X)$ be a non-zero $\vee$-closed ideal.
Let $f \in I$ be any non-zero element.
Then $f^{-1}f \in I$ since $I$ is an ideal.
From the description of idempotents in $I(X)$, the idempotent $f^{-1}f = 1_{A}$
for some non-empty subset $A \subseteq X$.
Let $x_{i} \in A$ for some $i$.
Then $1_{x_{i}} \in I$ since $I$ is an ideal.
Let $x_{j} \neq x_{i}$.
Let $g$ be the partial bijection that maps $x_{i}$ to $x_{j}$.
Then $1_{x_{j}} = g^{-1} 1_{x_{i}}g$.
It follows that $1_{x_{j}} \in I$.
Hence $I$ contains all idempotents defined on one-element subsets.
Since $I$ is closed under joins, $I$ contains all idempotents of $I(X)$.
Let $g$ be the partial bijection that sends $x$ to $y$ where $x,y \in X$.
Then $g = g1_{x}$.
But again, because $I$ is an ideal, we have that $g \in I$.
Thus $I$ contains all the elements of $I(X)$ with domain (and range) containing only one element.
But every element of $I(X)$ is a finite disjoint union of such elements.
It follows that $I = I(X)$, as claimed.}
\end{example}

We shall need some notation from the theory of posets.
Let $P$ be a poset.
If $X \subseteq P$, 
define
$$X^{\uparrow} = \{ y \in P \colon \exists x \in X \mbox{ such that } x \leq y \}$$
and
$$X^{\downarrow} = \{ y \in P \colon \exists x \in X \mbox{ such that } y \leq x \}.$$
If $X = X^{\uparrow}$ we say that $X$ is {\em closed upwards}.

We shall be working with inverse monoids that have binary meets and binary compatible joins.
The following is proved as Lemma~1.4.11 of \cite{Lawson98}.
It shows circumstances in which meets exist in an arbitrary inverse semigroup.

\begin{lemma}\label{le: meets} If $s \sim t$ then $s \wedge t$ exists and 
$$\mathbf{d}(s \wedge t) = \mathbf{d}(s)\mathbf{d}(t)
\mbox{ and }
\mathbf{r}(s \wedge t) = \mathbf{r}(s)\mathbf{r}(t).$$
\end{lemma}

The above lemma will often be used in the following situation:
if $s$ and $t$ are bounded above by an element then they are compatible.
It follows by the above lemma that $s \wedge t$ exists.

\begin{lemma}\label{le:rik} Let $S$ be a distributive inverse semigroup.
If $a \vee b$ exists then
$$\mathbf{d}(a \vee b) = \mathbf{d}(a) \vee \mathbf{d}(b)
\mbox{ and }
\mathbf{r}(a \vee b) = \mathbf{r}(a) \vee \mathbf{r}(b).$$
\end{lemma}

The following is just the finitary version of \cite{Res3}.

\begin{lemma}\label{lem:meets_and_joins}
Let $S$ be a distributive inverse monoid.
Suppose that $a \vee b$ and $c \wedge (a \vee b)$ both exist.
Then $c \wedge a$ and $c \wedge b$ both exist, 
the join $(c \wedge a) \vee (c \wedge b)$ exists and
$$c \wedge (a \vee b)
=
(c \wedge a) \vee (c \wedge b).$$
\end{lemma}
\begin{proof} We begin with two auxiliary results.

Suppose that $x \wedge y$ exists.
We prove that $x  \mathbf{d} (y) \wedge y$ exists and that $(x \wedge y)\mathbf{d} (y) = x\mathbf{d} (y) \wedge y$.
It is immediate that $(x \wedge y)\mathbf{d} (y) \leq x \mathbf{d}(y), y$.
Suppose now that $u \leq x\mathbf{d} (y), y$.
Then $u \leq x,y$ and so $u \leq x \wedge y$.
Thus $u\mathbf{d} (y) \leq (x \wedge y)\mathbf{d} (y)$.
But $u \leq y$ implies that $u = \mathbf{r} (u)y$ and so $u \mathbf{d} (y) = u$.
Hence $u \leq  (x \wedge y)\mathbf{d} (y)$.
It follows that  $(x \wedge y)\mathbf{d} (y) = x\mathbf{d} (y) \wedge y$, as claimed.

Suppose that $x \vee y$ exists.
We prove that $(x \vee y)\mathbf{d} (x) = x$.
Clearly, $x \leq (x \vee y)\mathbf{d} (x)$.
But $\mathbf{d} ((x \vee y)\mathbf{d} (x)) = \mathbf{d} (x \vee y)\mathbf{d} (x) = \mathbf{d} (x)$
since $\mathbf{d} (x) \leq \mathbf{d} (x \vee y)$.
But two elements bounded above having the same domain are equal.

We show that $c \wedge a$ exists.
Let $x \leq a,c$.
Then $x \leq a \vee b, c$ and so $x \leq c \wedge (a \vee b)$.
Thus $x \mathbf{d} (a) \leq (c \wedge (a \vee b))\mathbf{d} (a)$.
But $x\mathbf{d} (a) = x$.
It follows that $x \leq  (c \wedge (a \vee b))\mathbf{d} (a)$.
But $ (c \wedge (a \vee b))\mathbf{d} (a) \leq c,a$.
Thus $c \wedge a = (c \wedge (a \vee b))\mathbf{d} (a)$.
By symmetry, $c \wedge b$ exists.

Now $c \wedge a \leq a$ and $c \wedge b \leq b$ and so
$c \wedge a, c \wedge b \leq a \vee b$.
It follows that $c \wedge a \sim c \wedge b$ and so $(c \wedge a) \vee (c \wedge b)$ exists.

Observe that $(c \wedge a) \vee (c \wedge b) \leq a \vee b, c$.
Now let $x \leq c, a \vee b$.
From $x \leq a \vee b$ we get that $x = (a \vee b)\mathbf{d} (x)$.
From $x \leq a \vee b$ we get that $x \mathbf{d} (a) \leq (a \vee b)\mathbf{d} (a) = a$.
Thus $x\mathbf{d} (a) \leq c\mathbf{d} (a)$.
Hence $x \mathbf{d} (a) \leq c\mathbf{d} (a) \wedge a$.
Similarly, $x \mathbf{d} (b) \leq c\mathbf{d} (b) \wedge b$.
It follows that $x\mathbf{d} (a) \vee x\mathbf{d} (b) \leq (c\mathbf{d} (a) \wedge a) \vee (c \mathbf{d} (b) \wedge b)$.
We therefore get that $x \leq (c \wedge a)\mathbf{d} (a) \vee (c \wedge b)\mathbf{d} (b) \leq (c \wedge a) \vee (c \wedge b)$,
as required.
\end{proof}

Homomorphisms between Boolean inverse monoids will be unital, map zero to zero and preserve any binary compatible joins.
If the $\wedge$-operation is preserved we will call it a {\em morphism}.
The {\em kernel} of such a morphism is the preimage of zero.
The proof of the following is straightforward.

\begin{lemma} Let $\theta \colon S \rightarrow T$ be a morphism of Boolean inverse $\wedge$-monoids.
Then the kernel of $\theta$ is a $\vee$-ideal.
\end{lemma}

\begin{remark} {\em Ganna Kudryavtseva observed that the kernel of a morphism determines the morphism.
See the comments before Theorem~4.18 of \cite{Lawson12}.
It follows that the notion of $0$-simplifying is a genuine notion of simplicity.}
\end{remark}

Finite Boolean inverse $\wedge$-monoids are all known. 
See \cite{Lawson12} for a proof of the following.
The statement of part (1) below will be clarified later.
This theorem will play an important role in motivating our work.

\begin{theorem}\label{them:bongo} Let $S$ be a finite Boolean inverse $\wedge$-monoid.
\begin{enumerate}

\item There is a finite discrete groupoid $G$ such that $S$ is isomorphic to the set of all local bisections of $G$.

\item The fundamental such semigroups are the finite direct products of symmetric inverse m onoids.

\item The $0$-simplifying, fundamental such semigroups are the finite symmetric inverse monoids.

\end{enumerate}
\end{theorem}

The following relation was introduced in \cite{Lenz} and will be important in handling $0$-simplifying monoids.
Let $e$ and $f$ be two non-zero idempotents in $S$.
Define $e \preceq f$ if and only if there exists a set of  elements $X = \{x_{1}, \ldots, x_{m}\}$ such that
$e = \bigvee_{i=1}^{m}  \mathbf{d}(x_{i})$
and 
$\mathbf{r}(x_{i}) \leq f$ for $1 \leq i \leq m$. 
We can write this formally as $e = \bigvee \mathbf{d}(X)$ and $\bigvee \mathbf{r}(X) \leq f$.
We say that $X$ is a {\em pencil} from $e$ to $f$.
Define the relation $e \equiv f$ if and only if $e \preceq f$ and $f \preceq e$.
The following was proved as part of Lemma~7.8 of \cite{Lenz}.

\begin{lemma}\label{lem:toby} Let $S$ be a Boolean inverse $\wedge$-monoid.
The relation $\equiv$ is the universal relation on the set of non-zero idempotents if and only if $S$ is $0$-simplifying.
\end{lemma}
\begin{proof}
Suppose that $\equiv$ is the universal relation on the set of non-zero idempotents.
Let $0 \neq I \subseteq S$ where $I$ is a $\vee$-ideal.
Clearly there is a non-zero idempotent $e \in I$.
Let $f$ be any non-zero idempotent in $S$.
Then $f \equiv e$ and so, in particular,  $f \preceq e$.
It follows that there are elements $x_{i}$, where $1 \leq i \leq n$,
where $\mathbf{r}(x_{i}) \leq e$ and $f = \bigvee_{i=1}^{n} \mathbf{d}(x_{i})$.
We have that $ex_{i} = x_{i}$ and so $x_{i} \in I$.
But then $\mathbf{d}(x_{i}) \in I$.
But $I$ is a $\vee$-ideal and so $f \in I$, as required.

Before we prove the converse, observe that if $I$ is any ideal of $S$ then
the set $I^{\vee}$ consisting of all finite joins of non-empty compatible subsets of $I$ is a $\vee$-ideal.
Let $S$ be $0$-simplifying.
Let $e$ and $f$ be non-zero idempotents. 
By assumption, $(SeS)^{\vee} = S$.
Then $f = \bigvee_{i=1}^{m} e_{i}$ form some $e_{i} \in SeS$  where $1 \leq i \leq m$.
Let $e_{i} = y_{i}ex_{i}$.
Then $e_{i} = e_{i}y_{i}ex_{i}e_{i}$.
Put $e_{i}y_{i}e = u_{i}$ and $ex_{i}e_{i} = v_{i}$.
We have that $u_{i} = u_{i}v_{i}u_{i}$ and $v_{i}u_{i}v_{i} = v_{i}$.
Thus $v_{i} = u_{i}^{-1}$.
Then 
$e_{i} = v_{i}^{-1}v_{i}$ 
and
$v_{i}v_{i}^{-1} \leq e$.
We have therefore proved that $f \preceq e$.
\end{proof}

A {\em Tarski monoid} is a countable Boolean inverse $\wedge$-monoid whose semilattice of idempotents is a Tarski algebra.
Such monoids arise naturally as we now show.
The following was suggested by the first line in the proof of \cite[Theorem~6.11]{Matui12}.

\begin{proposition}\label{prop:blop} Let $S$ be a countable Boolean inverse $\wedge$-monoid.
If $S$ is $0$-simplifying then either $E(S)$ is the Tarski algebra or $E(S)$ is finite.
\end{proposition}
\begin{proof} We prove first that if $e$ is an atom and $e \, \mathscr{D} \, f$ then $f$ is an atom.
Let $e \stackrel{a}{\rightarrow} f$.
Suppose that $i \leq f$.
Then $ia \leq a$.
Hence $\mathbf{d}(ia) \leq e$.
Since $e$ is an atom, it follows that $\mathbf{d}(ia) = e$ or  $\mathbf{d}(ia) = 0$.
If $\mathbf{d}(ia) = 0$ then $ia = 0$ and so $i = 0$.
If  $\mathbf{d}(ia) = e$ then $ia = a$ and $i = f$.
It follows that $f$ is an atom.

If $E(S)$ contains no atoms it is a Tarski algebra since it is countable and cannot be finite.
We may therefore suppose that $E(S)$ contains at least one atom $e$.

Let $f$ be any non-zero idempotent.
By assumption $f \preceq e$.
There are therefore a finite number of {\em non-zero} elements $x_{1}, \ldots, x_{m}$ of $S$ such that
$f = \bigvee_{i=1}^{m} \mathbf{d}(x_{i})$ and $\mathbf{r}(x_{i}) \leq e$.
But $e$ is an atom and so $\mathbf{r}(x_{i}) = e$ for $i = 1, \ldots, m$.
But then $\mathbf{d}(x_{i})$ is an atom.
We have therefore proved that each non-zero element of $E(S)$ is the join of a finite number of atoms,
and all the atoms are $\mathscr{D}$-related to $e$.
It is an immediate consequence that any atom in $E(S)$ is $\mathscr{D}$-related to $e$.
Thus all the atoms of $E(S)$ form a single $\mathscr{D}$-class.
Since the identity is an idempotent it is the join of a finite number of atoms, say $e_{1}, \ldots, e_{m}$.
Then by distributivity, every non-zero idempotent is a join of some of these $m$ idempotents.
It follows that there are exactly $m$ atoms and the Boolean algebra of idempotents is finite with $2^{m}$ elements.
\end{proof}

If $S$ is fundamental then it may be embedded in the Munn semigroup of its semilattice of idempotents.
The proof of the following is now immediate.

\begin{corollary}\label{cor:judy} Let $S$ be a fundamental countable Boolean inverse $\wedge$-monoid.
If $S$ is $0$-simplifying then either $S$ is a Tarski monoid or $S$ is a finite symmetric inverse monoid.
\end{corollary}

\begin{remark}{\em  The above corollary suggests that $0$-simplifying, fundamental Tarski monoids are natural generalizations of finite symmetric inverse monoids.
This will guide us in developing the theory of such monoids.}
\end{remark}

\section{A non-commutative dictionary}

The goal of this section is to describe how non-commutative Stone duality enables us to construct
a dictionary between Boolean inverse $\wedge$-monoids and a class of \'etale groupoids.
We refer to \cite{Lawson10b, Lawson12} for any proofs.

\subsection{Non-commutative Stone duality}
Let $G$ be a groupoid.
We denote its set of identities by $G_{o}$ and the domain and range maps by $\mathbf{d}$ and $\mathbf{r}$, respectively.
It is said to be a {\em topological groupoid} if the groupoid multiplication and the maps $\mathbf{d}$, $\mathbf{r}$ and inversion are continuous.
It is said to be an {\em \'etale (topological) groupoid} if the maps $\mathbf{d}$ and $\mathbf{r}$ are local homeomorphisms.
It is the \'etale property that is crucial for the connections between \'etale groupoids and semigroups
since it implies that the open subsets of $G$ form a monoid under multiplication of subsets of $G$ \cite[Chapter 1]{Res1}.
We shall be interested in \'etale groupoids where we impose additional conditions on the space of identities.

Let $S$ be a Boolean $\wedge$-monoid.
A {\em filter} in $S$ is a subset $A \subseteq S$ which is closed under finite meets and closed upwards.
It is said to be {\em proper} if $0 \notin A$.
A set $X \subseteq S$ is called a {\em filter base} if for all $a,b \in X$ there exisrs $c \in X$ such that $c \leq a,b$.
The proof of the following is straightforward.

\begin{lemma} Let $X$ be a filter base in a Boolean inverse semigroup.
Then $X^{\uparrow}$ is a filter.
\end{lemma}

A maximal proper filter is called an {\em ultrafilter}.
Ultrafilters may be characterized amongst proper filters by means of the following.
If $A$ is a filter and $s \in S$ we write  $s \wedge A \neq 0$ to mean $s \wedge a \neq 0$ for all $a \in A$.

\begin{lemma}\label{lem:exel} Let $S$ be a Boolean inverse $\wedge$-monoid.
A proper filter $A$ is an ultrafilter if, and only if, $s \wedge A \neq 0$ implies that $ s \in A$.
\end{lemma}

A proper filter $A$ is said to be {\em prime} if $a \vee b \in A$ implies that $a \in A$ or $b \in A$.

\begin{lemma}\label{lem:molina} Let $S$ be a Boolean inverse $\wedge$-monoid.
A proper filter is prime if, and only if, it is an ultrafilter.
\end{lemma}

The set of all ultrafilters of $S$ is denoted by $\mathsf{G}(S)$.
If $A$ is an ultrafilter, define
$$
\mathbf{d}(A) = (A^{-1}A)^{\uparrow}
\mbox{ and }
\mathbf{r}(A) = (AA^{-1})^{\uparrow},
$$
both ultrafilters.
Define a partial binary operation $\cdot$ on  $\mathsf{G}(S)$ by
$$A \cdot B = (AB)^{\uparrow}$$
if $\mathbf{d}(A) = \mathbf{r}(B)$, and undefined otherwise.
Then this is well-defined and $(\mathsf{G}(S), \cdot)$ is a groupoid.
Those ultrafilters that are identities in the groupoid are called {\em idempotent ultrafilters}.
They are precisely the ultrafilters that are also inverse submonoids.
We denote this set by  $\mathsf{G}(S)_{o}$.
If $F \subseteq E(S)$ is an ultrafilter then $F^{\uparrow}$ is an idempotent ultrafilter in $S$ and every idempotent ultrafilter is of this form.
If $G$ is an idempotent ultrafilter in $S$ and $a \in S$ is such that $a^{-1}a \in G$ then $A = (aG)^{\uparrow}$ is an ultrafilter
where $\mathbf{d}(A) = G$, and every ultrafilter in $S$ is constructed in this way.
Denote by $V_{a}$ the set of all ultrafilters in $S$ that contain the element $a$.

\begin{lemma}\label{lem:pele} Let $S$ be a Boolean inverse $\wedge$-monoid.
\begin{enumerate}
 
\item $V_{a} \subseteq V_{b}$ if and only if $a \leq b$.

\item $V_{a}V_{b} = V_{ab}$.

\item $V_{a} \cap V_{b} = V_{a \wedge b}$.

\item If $a \vee b$ exists then $V_{a} \cup V_{b} = V_{a \vee b}$.

\item $V_{a}$ consists of idempotent ultrafilters if and only if $a$ is an idempotent.

\end{enumerate}
\end{lemma}

Put $\tau = \{V_{a} \colon a \in S\}$.
Then $\tau$ is the basis for a topology on $\mathsf{G}(S)$
with respect to which it is a Hausdorff \'etale topological groupoid
such that  $\mathsf{G}(S)_{o}$ is Hausdorff, compact and has  a basis of clopen subsets.
We call any topological groupoid satisfying these properties a {\em Boolean groupoid}.
If $G$ is an arbitrary groupoid a subset $X \subseteq G$ is called a {\em local bisection} if $X^{-1}X,XX^{-1} \subseteq G_{o}$.
If $G$ is a Boolean groupoid, we denote by $\mathsf{B}(G)$ the set of all compact-open local bisections of $G$.
This is a Boolean inverse $\wedge$-monoid.

\begin{theorem}[Non-commutative Stone duality]\label{them:duality} For suitable definitions of morphisms,
the category of Boolean inverse $\wedge$-monoids is dually equivalent to the category of Boolean groupoids
under the functors $S \mapsto \mathsf{G}(S)$ and $G \mapsto \mathsf{B}(G)$.
\end{theorem}

If $S$ is a Boolean inverse $\wedge$-monoid, then $\mathsf{X}(S)$ denotes the topological space associated with the Boolean algebra $E(S)$.
It is homeomorphic to $\mathsf{G}(S)_{o}$ and we call it the {\em structure space} of $S$. 
If $e \in E(S)$, we denote by $U_{e}$ the set of all ultrafilters in $E(S)$ that contain $e$.

\subsection{Units and infinitesimals}

The spatial realization theorem deals with the relationship between the group of units of an inverse monoid and the whole inverse monoid.
We therefore need a way of constructing units.
The starting point is provided by the following lemma.

\begin{lemma}\label{lem:spooks} Let $S$ be a Boolean inverse $\wedge$-monoid.
\begin{enumerate}

\item Let $s \in S$ such that $s^{-1}s = ss^{-1}$.
Put $e = s^{-1}s$.
Then $g = s \vee \bar{e}$ is invertible.

\item  $a^{2} = 0$ if, and only if, $a^{-1}a \perp aa^{-1}$  if, and only if, $a \perp a^{-1}$.

\item If $a^{2} = 0$ and $e = \overline{a^{-1}a} \,   \overline{aa^{-1}}$ then
$$u = a^{-1} \vee a \vee e$$
is a non-trivial involution lying above $a$.

\end{enumerate}
\end{lemma}
\begin{proof} (1) Simply observe that $g^{-1}g = gg^{-1} = s^{-1}s \vee \overline{s^{-1}s} = 1$.

(2) If $a^{2} = 0$ then $a^{-1}aaa^{-1} = 0$ and so $a^{-1}a \perp aa^{-1}$.
If  $a^{-1}a \perp aa^{-1}$ then $a^{-1}aaa^{-1} = 0$ and so $a^{2} = 0$.
The equivalence of $a^{-1}a \perp aa^{-1}$  with $a \perp a^{-1}$ is immediate.

(3) The elements $a$ and $a^{-1}$ are orthogonal.
Put $s = a \vee a^{-1}$.
Then $s^{-1}s = ss^{-1}$.
Now apply part (1).
It is straightforward to check that it is an involution.
\end{proof}

In calculating with infinitesimals, it becomes essential to draw pictures such as the following where the arrow signifies the direction in which the diagram should be read.
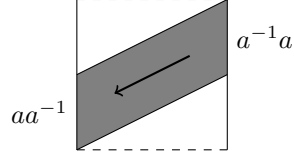
\begin{center}
\begin{tikzpicture}
\draw [fill=gray] (0,0) -- (2,1) -- (2,2) -- (0,1) -- (0,0);

\draw (0,1) -- (0,2);
\draw (2,0) -- (2,1);

\draw [dashed]  (0,2) -- (2,2); 
\draw [dashed]  (0,0) -- (2,0); 

\node [right] at (2,1.5) {$a^{-1}a$};
\node [left] at (0,0.5) {$aa^{-1}$};

\draw [<-,thick] (0.5,0.75) -- (1.5,1.25);

\end{tikzpicture}
\end{center}
\begin{remark} {\em We may paraphrase the above results by saying that every element of $S$ in  the group of units of a local submonoid is beneath an element of the group of units,
as is  every infinitesimal.}
\end{remark}

\begin{example} {\em Let $I_{n}$ be a finite symmetric inverse monoid on $n$ letters.
Examples of infinitesimal elements are those elements of the form $x \mapsto y$ where $x,y \in X$ and $x \neq y$.
The group elements associated with these, as constructed in the above lemma, are precisely the transpositions.}
\end{example}

Our next result is fundamental since it enables us to construct infinitesimals with specific properties.

\begin{proposition}\label{prop:george} Let $S$ be a $0$-simplifying Tarski monoid.
Let $F \subseteq E(S)$ be an ultrafilter and let $e \in F$.
Then there exists an element $a \in S$ such that
\begin{enumerate}

\item $a$ is an infinitesimal.

\item $a^{-1}a \in F$.

\item $a \in eSe$.

\end{enumerate}
\end{proposition}
\begin{proof}
The idempotent $e \neq 0$.
We are working in a Tarski algebra, and so $e$ cannot be an atom.
Thus there exists $0 \neq f < e$.
The idempotents form a Boolean algebra, and so $e = f \vee \bar{f}$ and $f \wedge \bar{f} = 0$.
Since $f \vee \bar{f} = e \in F$, and $F$ is an ultrafilter and so a prime filter, we know that $f \in F$ or $\bar{f} \in F$.
Without loss of generality, we may assume that $f \in F$.
Now $S$ is $0$-simplifying and so $\bar{f} \equiv f$.
In particular, $f \preceq \bar{f}$.
We may therefore find elements $x_{1}, \ldots, x_{m}$ 
such that
$f = \bigvee_{i=1}^{m} \mathbf{d}(x_{i})$ 
and $\mathbf{r}(x_{i}) \leq \bar{f}$.
We use the fact that $F$ is a prime filter, to deduce that $\mathbf{d}(x_{i}) \in F$ for some $i$.
Put $a = x_{i}$.
Then $a^{-1}a \leq f$ and $aa^{-1} \leq \bar{f}$.
Hence $a^{-1}a \perp aa^{-1}$.
It follows that $a$ is an infinitesimal.
Clearly,  $a^{-1}a,aa^{-1} \leq e$.
In addition, $a^{-1}a \in F$.
\end{proof}

The above proposition tells us that infinitesimals are plentiful in $0$-simplifying Tarski monoids.
As a result, by Lemma~\ref{lem:spooks} involutions are plentiful.

\begin{lemma}\label{lem:horsa} Let $S$ be a $0$-simplifying Tarski monoid. 
Let $e$ be any non-zero idempotent.
Then there exist infinitesimals $a,b \in eSe$ such that $ab$ is a restricted product and an infinitesimal.
\end{lemma}
\begin{proof} Every non-zero idempotent is an element of some ultrafilter in $E(S)$.
Thus by Proposition~\ref{prop:george}, we may find an infinitesimal $x \in eSe$.
Similarly, we may find an infinitesimal $b \in \mathbf{d}(x)S\mathbf{d}(x)$.
Put $a = x \mathbf{r}(b)$.
The set of infinitesimals forms an order ideal, and so $a$ is an infinitesimal.
By construction, $ab$ is a restricted product, and since $\mathbf{r}(a) \perp \mathbf{d}(b)$ it is an infinitesimal. 
\end{proof}

\begin{lemma}\label{lem:hengist} Let $S$ be a $0$-simplifying Tarski monoid. 
Every ultrafilter contains an infinitesimal or the product of two infinitesimals.
\end{lemma}
\begin{proof} We may restrict our attention to non-idempotent ultrafilters $A$.

Suppose first that  $A$ is an ultrafilter such that $A^{-1} \cdot A \neq A \cdot A^{-1}$.
Both  $A^{-1} \cdot A$ and $A \cdot A^{-1}$ are idempotent ultrafilters and are distinct by assumption.
Since the groupoid $\mathsf{G}(S)$ is Hausdorff there are compact-open sets $V_{s}$ and $V_{t}$ such that
$A^{-1} \cdot A \in V_{s}$ and $A \cdot A^{-1} \in V_{t}$ where $s \wedge t = 0$.
We may find idempotents $e$ and $f$ such that $A^{-1} \cdot A \in V_{e}$ and $A \cdot A^{-1} \in V_{f}$  and $e \wedge f = 0$.
Let $a \in A$ and put $b = fae$.
Then $b \in A$ and $b^{2} = 0$.

We now consider the case where $A$ be an ultrafilter such that $A^{-1} \cdot A = A \cdot A^{-1} = F^{\uparrow}$ where $F \subseteq E(S)$ is an ultrafilter.
We shall prove that there is an ultrafilter $G \subseteq E(S)$ distinct from $F$ and an ultrafilter $B$ such that $F^{\uparrow} = B \cdot B^{-1}$ and $B^{-1} \cdot B = G^{\uparrow}$.
Then $B^{-1} \cdot B \neq B \cdot B^{-1}$ and $A = (A \cdot B) \cdot B^{-1}$.
By  the first case above, both $A \cdot B$ and $B^{-1}$ contain infinitesimals and so $A$ contains a product of infinitesimals.

Let $F \subseteq E(S)$ be an ultrafilter.
Let $e \in F$.
Using the fact that $E(S)$ is a Tarski algebra, we may write $e = e_{1} \vee e_{2}$ where $e_{1},e_{2} \neq 0$ and $e_{1} \perp e_{2}$.
Without loss of generality, we may assume that $e_{1} \in F$ and $e_{2} \notin F$.
We now relabel.
Let $e \in F$ and let $f \neq 0$ be such that $e \perp f$.
By assumption, $e \preceq f$.
Thus there are elements $x_{1}, \ldots, x_{m}$ such that
$e = \bigvee_{i=1}^{m} \mathbf{d}(x_{i})$ and $\mathbf{r}(x_{i}) \leq f$.
Since $F$ is an ultrafilter it is also a prime filter and so, relabelling if necessary, $\mathbf{d}(x_{1}) \in F$.
Consider the ultrafilter $C = (x_{1}F^{\uparrow})^{\uparrow}$.
Then $\mathbf{d}(C) = F^{\uparrow}$.
Put $G = E(\mathbf{r}(C))$.
Then $f \in G$.
It follows that $C \cdot C^{-1} \neq F^{\uparrow}$.
\end{proof}

The following is key.
It follows from Lemma~\ref{lem:hengist} and Lemma~\ref{lem:spooks}.

\begin{proposition}\label{prop:dinky}  Let $S$ be a $0$-simplifying Tarski monoid. 
Then every ultrafilter contains a unit.
\end{proposition}

We shall now transform the above result, which is essentially a property of the associated groupoid,
into a visible property of the inverse monoid.

An inverse monoid is said to be {\em factorizable} if every element lies beneath an element of the group of units.
Finite symmetric inverse monoids are factorizable.
Further discussion of the applications of this concept may be found in \cite{Lawson98}.
In this paper, we shall need a weaker notion.
A distributive inverse monoid $S$ is said to be {\em piecewise factorizable} if each element $s \in S$ may 
be written in the form $s = \bigvee_{i=1}^{m} s_{i}$ where for each $s_{i}$ there is a unit $g_{i}$ such that $s_{i} \leq g_{i}$.
This may be rewritten in the following, more striking form:
$$s = \bigvee_{i=1}^{m} g_{i}s^{-1}s.$$
In the factorizable case, we have that $s = gs^{-1}s$, which explains our choice of terminology.

\begin{lemma}\label{lem:welsh} Let $S$ be a Boolean inverse $\wedge$-monoid.
Then $S$ is piecewise factorizable if, and only if, each ultrafilter of $S$ contains a unit.
\end{lemma}
\begin{proof} Suppose first that $S$ is piecewise factorizable.
Let $A$ be any ultrafilter and choose $s \in A$.
Then by assumption we may write 
$s = \bigvee_{i=1}^{m} s_{i}$ where for each $s_{i}$ there is a unit $g_{i}$ such that $s_{i} \leq g_{i}$.
But every ultrafilter is prime and so $s_{i} \in A$ for some $i$.
It is now immediate that $g_{i} \in A$ and so each ultrafilter contains a unit.
To prove the converse, we assume that every ultrafilter contains a unit.
Let $s \in S$ be any non-zero element.
We shall write $V_{s}$ as a union of clopen sets.
Let $A \in V_{s}$.
Then there is some unit $g \in A$.
Thus $g \wedge s \in A$.
We may therefore write
$V_{s} = \bigcup V_{s_{i}}$ where the $s_{i}$ are those elements belonging to the elements of $V_{s}$ which are  beneath units.
By compactness, we may write $V_{s} = \bigcup_{i=1}^{m} V_{s_{i}}$ and the result follows by Lemma~\ref{lem:pele}.
\end{proof}

If we combine Proposition~\ref{prop:dinky} and Lemma~\ref{lem:welsh}, we obtain our first structural result about $0$-simplifying Tarski monoids.

\begin{proposition}\label{prop:dory}  
Every $0$-simplifying Tarski monoid is piecewise factorizable.
\end{proposition}

This result hints at the very close connection between the structure of such monoids and the structure of their groups of units.

Let $G$ be a group.
We denote by $\mathsf{K}(G)$ the set of cosets of subgroups of $G$.
We refer to the elements of $\mathsf{K}(G)$ simply as {\em cosets}.
It is well-known that these are precisely the non-empty subsets $A\subseteq G$ such that $A = AA^{-1}A$.
In fact,  $\mathsf{K}(G)$ is an inverse monoid, but we shall only treat it as a groupoid.
The groupoid product is defined by $A \cdot B$ is defined if $A^{-1}A = BB^{-1}$ in which case $A \cdot B = AB$.
The identities of this groupoid are the subgroups of $G$.

\begin{lemma}\label{lem:sunshine} Let $S$ be a Boolean inverse $\wedge$-monoid with group of units $G$.
Let $A$ be an ultrafilter in $S$. 
If $A \cap G$ is non-empty then it is a coset.
\end{lemma}
\begin{proof}
Put $X = A \cap G$.
The claim is proved once we have shown that $X = XX^{-1}X$
Only one inclusion needs proving.
Let $g,h,k \in X$.
Then $a \leq g$, $b \leq h$ and $c \leq k$ where $a,b,c \in A$.
Then $ab^{-1}c \leq gh^{-1}k$ and $ab^{-1}c \in A$.
The result follows.
\end{proof}

Our next result shows that viewed as a discrete groupoid, 
$\mathsf{G}(S)$ has a local structure very similar to that of the groupoid $\mathsf{K}(G)$.

\begin{proposition}\label{prop: covering_functor} 
Let $S$ be a piecewise factorizable Boolean inverse $\wedge$-monoid with group of units $G$.
Define $\gamma \colon \: \mathsf{G}(S) \rightarrow \mathsf{K}(G)$ by $A \mapsto A \cap G$.
Then $\gamma$ is a covering functor.
\end{proposition}
\begin{proof} The map is well-defined by Lemma~\ref{lem:sunshine}.
Let $F$ be an identity in $\mathsf{G}(S)$.
Then $F$ is an inverse subsemigroup of $S$.
It follows that $F \cap G$ is a subgroup and so $\gamma (F)$ is an identity of $\mathsf{K}(G)$.
Thus $\gamma$ maps identities to identities.
Also, $\gamma (A^{-1}) = \gamma (A)^{-1}$.
It is immediate that $\gamma (A)^{-1} \cdot \gamma (A) \subseteq \gamma (A^{-1} \cdot A)$.
To prove the reverse inclusion, let $a \in \gamma (A^{-1} \cdot A)$.
Then $a \in A^{-1} \cdot A \cap G$.
It follows that there is an idempotent $e \in A^{-1} \cdot A$ such that $e \leq a$.
Let $g \in A \cap G$ be arbitrary.
Then $a = (ag^{-1})g$ where $ag^{-1} \in A^{-1}$.
 It follows that $a \in \gamma (A)^{-1} \cdot \gamma (A)$, as required.
Suppose now that $A \cdot B$ is defined.
Then, by the above, $\gamma (A) \cdot \gamma (B)$ is defined.
It is immediate that $\gamma (A) \cdot \gamma (B) \subseteq \gamma (A \cdot B)$.
Let $g \in \gamma (A \cdot B)$.
Then $g \in A \cdot B \cap G$.
Let $h \in b \cap G$ be arbitrary.
Then $gh^{-1} \in A \cap G$.
It follows that $g = (gh^{-1})h \in \gamma (A) \gamma (B)$, as required.
We have therefore shown that $\gamma$ is a functor.

By Lemma~2.11 of \cite{Lawson10b}, we have the following result.
Let $A$ and $B$ be two ultrafilters in an inverse monoid $S$ with group of units $G$  
such that $A^{-1} \cdot A = B^{-1} \cdot B$ and $A \cap G \cap B \neq \emptyset$.
Then $A = B$.
It follows from this that  the functor $\gamma$ is star injective.

It remains to show that $\gamma$ is star surjective.
Let $F$ be an idempotent ultrafilter in $\mathsf{G}(S)$.
Let $A \in \mathsf{K}(G)$ be a coset such that $A^{-1} \cdot A = \gamma (F)$.
Let $g \in A$ be arbitrary.
Then $B = (gF)^{\uparrow}$ is an ultrafilter in $S$ such that $B^{-1} \cdot B = F$ and $\gamma (B) = A$.
 \end{proof}

\subsection{The fundamental case}

We shall now define an important representation of Boolean inverse $\wedge$-monoids.
It is based on the following observation which is a well-known construction in the theory of groupoids.
Let $A \subseteq G$ be a local bisection.
Then we may define a bijection $A^{-1}A \rightarrow AA^{-1}$ by $a^{-1}a \mapsto aa^{-1}$ where $a \in A$.
Let $S$ be a Boolean inverse $\wedge$-monoid.
We shall define a homomorphism $\theta \colon S \rightarrow I(\mathsf{X}(S))$.
For each $s \in S$, the set $V_{s}$ is a compact-open local bisection of $\mathsf{G}(S)$.
The set $V_{\mathbf{d}(s)}$ consists of idempotent ultrafilters.
If $A$ is an idempotent ultrafilter then it is equal to $E(A)^{\uparrow}$ where $E(A)$ is an ultrafilter
in the Boolean algebra $E(S)$.
We define $\theta_{s} \colon U_{s^{-1}s} \rightarrow U_{ss^{-1}}$ as follows.
Let $F \in U_{s^{-1}s}$.
Then $A = (sF)^{\uparrow}$ is an ultrafilter containing $s$  with $\mathbf{d}(A) = F^{\uparrow}$.
We define
$$\theta_{s}(F) = E(\mathbf{r}(A)).$$
It is easy to check that this is well-defined and a homeomorphism.
It may be directly defined as
$$\theta_{s}(F) = E((sFs^{-1})^{\uparrow}) $$
whenever $F \in U_{s^{-1}s}$.

\begin{lemma}\label{lem:bingo} Let $S$ be a Boolean inverse $\wedge$-monoid.
\begin{enumerate}

\item $\theta_{s} = \theta_{t}$ if, and only if, $s\,  \mu \, t$.

\item If $g$ and $h$ are units, then $\theta_{g} = \theta_{h}$ if, and only if, $gFg^{-1} = hFh^{-1}$ for all
ultrafilters $F \subseteq E(S)$.

\end{enumerate}
\end{lemma}
\begin{proof} (1) It is immediate that if $s\,  \mu \, t$ then  $\theta_{s} = \theta_{t}$.
We prove the converse.
Suppose that  $\theta_{s} = \theta_{t}$.
Let $e \leq s^{-1}s$.
Then $U_{e} \subseteq U_{s^{-1}s}$.
Now verify that the image of $U_{e}$ under $\theta_{s}$ is $U_{ses^{-1}}$.
By a similar argument, we have that  the image of $U_{e}$ under $\theta_{t}$ is $U_{tet^{-1}}$.
Thus $U_{ses^{-1}} = U_{tet^{-1}}$.
But in a Boolean algebra, two elements are equal if, and only if, they are contained in the same sets of ultrafilters.
It follows that $ses^{-1} = tet^{-1}$.
We have therefore shown that $s \, \mu \,t$.

(2) This is immediate from the definitions.

\end{proof}

Being fundamental is a standard property of inverse semigroups.
We now link this to properties of the associated Boolean groupoid. 

Let $G$ be a groupoid.
The union of the local groups of $G$, denoted by $\mbox{Iso}(G)$, is a subgroupoid,
called the {\em isotropy subgroupoid} of $G$.
If $e \in G_{o}$ is such that the local group at $e$ is trivial we say that $e$ is {\em aperiodic}.
The groupoid $G$ is said to be {\em principal} if every identity is aperiodic.
That is, $\mbox{Iso}(G) = G_{o}$.
Equivalently, $G$ is an equivalence relation.
In the case where $G$ is a topological groupoid,
we say that $G$ is {\em essentially principal} if the interior of  $\mbox{Iso}(G)$ is just $G_{o}$.

\begin{proposition}\label{prop:one} Let $S$ be a Boolean inverse $\wedge$-monoid.
Then its associated Boolean groupoid $\mathsf{G}(S)$ is essentially principal
if and only if $S$ is fundamental.
\end{proposition}
\begin{proof} Suppose first that $\mathsf{G}(S)$ is essentially principal.
Let $a \in Z(E(S))$.
We need to prove that $a \in E(S)$.
Let $A \in V_{a}$.
We prove first that $A^{-1} \cdot A = A \cdot A^{-1}$.
We have that $A = (aA^{-1} \cdot A)^{\uparrow}$.
Now $A \cdot A^{-1} = (a A^{-1}A a^{-1})^{\uparrow}$.
Let $x \in  A \cdot A^{-1}$.
Then $aea^{-1} \leq x$ for some idempotent $e \in A^{-1} \cdot A$.
But by assumption, $a$ commutes with all idempotents.
Thus $aa^{-1}e \leq x$ and $aa^{-1} = a^{-1}a$.
Hence $a^{-1}ae \leq x$.
But $a^{-1}a,e \in A^{-1} \cdot A$ and so $a^{-1}ae \in A^{-1} \cdot A$.
It follows that $x \in A^{-1} \cdot A$.
We have therefore proved that $A \cdot A^{-1} \subseteq A^{-1} \cdot A$.
Now let $x \in   A^{-1} \cdot A$.
Then $e \leq x$ where $e \in A^{-1} \cdot A$ is an idempotent.
Clearly, $ea^{-1}a \leq x$ since $a^{-1}a \in A^{-1} \cdot A$.
But $ea^{-1}a = eaa^{-1} = aea^{-1}$.
It follows that $x \in A \cdot A^{-1}$.
We have therefore proved that $A \in \mbox{Iso}(\mathsf{G}(S))$.
But $A$ was arbitrary,
so we have proved that $V_{a} \subseteq \mbox{Iso}(\mathsf{G}(S))$.
Thus $V_{a}$ is contained in the interior of the isotropy subgroupoid.
Hence $V_{a} \subseteq \mathsf{G}(S)_{o}$.
It follows that every ultrafilter containing $a$ is an idempotent ultrafilter
which implies that $a$ is an idempotent, as required.

Conversely, assume that $S$ is fundamental.
Let $V_{a} \subseteq \mbox{Iso}(G)$.
We shall prove that $a \in Z(E(S))$.
Let $e \in E(S)$ be an arbitrary idempotent.
We claim that $V_{ea} = V_{ae}$.
Granting this we get that $ae = ea$.
We now use the fact that $S$ is fundamental to deduce that $a$ is an idempotent.
It follows that every ultrafilter in $V_{a}$ is idempotent.
Thus the interior of the isotropy groupoid is the space of identities.

It only remains, therefore, to prove the claim.
Let $A \in V_{ea}$.
Then $ea \leq a$ and so $a \in A$.
It follows that $A \in V_{a}$.
By assumption, $A^{-1} \cdot A = A \cdot A^{-1}$.
Now $ea \in A$ implies that $ea(ea)^{-1} \in A \cdot A^{-1}$.
Thus by assumption, $ea(ea)^{-1} \in A^{-1} \cdot A$.
Hence $aeaa^{-1}e \in A$ and so $ae \in A$.
We have show that $A \in V_{ae}$.
The reverse inclusion follows by symmetry.
\end{proof}

\subsection{The support operator}

The following notion and result is from \cite{Leech}.
Let $S$ be an inverse monoid.
A function $\phi \colon S \rightarrow E(S)$ is called a {\em fixed-point operator} if it satisfies the following two conditions:
\begin{description}

\item[{\rm (FPO1)}]  $s \geq \phi (s)$.

\item[{\rm (FPO2)}]  If $s \geq e$ where $e$ is any idempotent then $\phi (s) \geq e$.

\end{description}

\begin{proposition}\label{prop:leech} An inverse monoid $S$ is an inverse $\wedge$-monoid if, and only if, it possesses a fixed-point operator.
In an inverse $\wedge$-monoid, we have that $\phi (s) = s \wedge 1$.
\end{proposition}

In our work, it will be more convenient to work with 
$$\sigma (s) = \overline{\phi (s)}s^{-1}s$$
which we call the {\em support operator}.
The idempotent $\sigma (s)$ is called the {\em support of $s$}.

\begin{lemma}\label{lem:cooper} Let $S$ be a Boolean inverse $\wedge$-monoid.
Then
$$s = \phi (s) \vee s \sigma (s)$$
is an orthogonal join, and $\phi (s \sigma (s)) = 0$.
\end{lemma}
\begin{proof} Let $s \in S$.
Observe that $1 = \phi (s) \vee \overline{\phi (s)}$.
Multiplying on the right by $s^{-1}s$ and observing that $\phi (s) \leq s^{-1}s$,
we get that $s^{-1}s = \phi (s) \vee \sigma (s)$.
Multiplying on the left by $s$ and observing that $s \phi (s) = \phi (s)$,
we get that $s = \phi (s) \vee s \sigma (s)$.
It is routine to check that $\phi (s) \perp s \sigma (s)$, and that $\phi (s \sigma (s)) = 0$.
\end{proof}

Part (2) below is further evidence of the interaction between the properties of the group of units of the monoid
and the properties of the monoid as a whole.

\begin{lemma}\label{lem:jerry} Let $S$ be a Boolean inverse $\wedge$-monoid.
\begin{enumerate}

\item If $a,b \in S$ are such that $\overline{\phi (a)} \,   \overline{\phi (b)} = 0$ then $ab = ba$.

\item If $g,h \in U(S)$ and $\sigma (g)\sigma (h) = 0$ then $[g,h] = 1$.

\item Let $g$ and $h$ be units. Then $\sigma (ghg^{-1}) = g \sigma (h) g^{-1}$. 

\end{enumerate}
\end{lemma}
\begin{proof} (1) From $1 =  \phi (a) \vee \overline{\phi (a)}$ and $a = 1a1$.
It quickly follows that
$$a = \overline{\phi (a)} a \overline{\phi (a)} \vee \phi (a).$$
In addition, easy calculations show that
$\overline{\phi (a)}  \leq \phi (b)$ and $\overline{\phi (b)} \leq \phi (a)$.
We calculate
$$ab = \phi (a) \phi (b) \vee \overline{\phi (b)} b \overline{\phi (b)} \vee \overline{\phi (a)} a \overline{\phi (a)}.$$
By symmetry, this is equal to $ba$.

(2) This is immediate by (1), and the fact that when $g$ is a unit $\sigma (g) = \overline{\phi (g)}$.

(3) From $\phi (h) \leq h$ we get that $g \phi (h) g^{-1} \leq ghg^{-1}$.
From $\phi (ghg^{-1}) \leq ghg^{-1}$ we get that $g^{-1} \phi (ghg^{-1})g \leq h$ and so $g^{-1} \phi (ghg^{-1})g \leq \phi (h)$.
Thus $\phi (ghg^{-1}) \leq g \phi (h) g^{-1}$.
It follows that $g \phi (ghg^{-1})g^{-1} = \phi (ghg^{-1})$.
Take complements to get the desired result.
\end{proof}

\begin{lemma}\label{lem:jinxy} Let $S$ be a fundamental Boolean inverse $\wedge$-monoid.
\begin{enumerate}

\item Suppose that $af = fa$ for all $f \leq \overline{\phi (a)}$.
Then $a$ is an idempotent.

\item If $a^{-1}a = aa^{-1}$ and  $af = fa$ for all $f \leq \sigma (a) $.
Then $a$ is an idempotent.

\end{enumerate}
\end{lemma}
\begin{proof} (1 ) Let $e$ be an arbitrary idempotent.
Since $1 = \overline{\phi (a)} \vee \phi (a)$ we have that $e = e \overline{\phi (a)} \vee e \phi (a)$.
Put $i = e \overline{\phi (a)} \leq \overline{\phi (a)}$ and $j = e \phi (a) \leq \phi (a)$.
By assumption, $ai = ia$.
Clearly $j \leq a$ and so $j = aj = ja$.
Thus $a$ commutes with both $i$ and $j$.
It is immediate that $a$ commutes with $e$.
But $e$ we arbitrary and so $a$ commutes with any idempotent.
Since $S$ is fundamental, it follows that $a$ is an idempotent.

(2) This is immediate by (1).
\end{proof}

We give an explicit proof of the following because of its importance.

\begin{lemma}\label{lem:dream}
Let $F \in U_{s^{-1}s}$.
\begin{enumerate}

\item $(sFs^{-1})^{\uparrow}$ is an ultrafilter in $S$.

\item If $sFs^{-1} \subseteq F$ then  $F$ is the only ultrafilter in $E(S)$ containing $sFs^{-1}$.

\end{enumerate}
\end{lemma} 
\begin{proof} (1) It is easy to check that  $(sFs^{-1})^{\uparrow}$ is a proper filter.
To show that it is an ultrafilter, it is enough to prove the following.
let $e$ be an idempotent $e \leq ss^{-1}$ such that
$e(sfs^{-1}) \neq 0$ for all $f \in F$.
Then $s^{-1}esf \neq 0$ for all $f \in F$.
It follows that $s^{-1}es \in F$.
Thus $e \in sFs^{-1}$.

The proof of (2) follows immediately from the proof of (1).
\end{proof}

Our next result establishes that our abstract notion of support agrees with the concrete notion.
If $Y$ is a subset of a topological space then $\mathsf{cl}(Y)$ denotes the {\em closure} of that subset.

\begin{proposition}\label{prop:tempest} Let $S$ be a fundamental Boolean inverse $\wedge$-monoid.
For each $s$, we have that
$$U_{\sigma (s)} = \mathsf{cl}(\{ F \colon F \in U_{s^{-1}s} \mbox{ and } sFs^{-1} \nsubseteq F   \}).$$
\end{proposition}
\begin{proof}
Let $F$ be such that $F \in U_{s^{-1}s}$
and
$sFs^{-1} \nsubseteq F$.
If $\phi (s) \in F$ then 
$s \in F^{\uparrow}$ and so $sFs^{-1}  \subseteq F$, which is a contradiction.
Thus $\overline{\phi (s)} \in F$ which, together with the fact that $s^{-1}s \in F$ gives $\sigma (s) \in F$.

To prove the reverse inclusion, 
put 
$$Y = \{ F \colon  F \in U_{s^{-1}s} \mbox{ and }  sFs^{-1} \nsubseteq F   \}.$$ 
Let $F \in U_{\sigma (s)}$.
We need to show that every open set containing $F$ intersects $Y$.
It is enough to restrict attention to those open sets $U_{e}$ where $e \in F$.
Suppose that $Y \cap U_{e} = \emptyset$.
Observe that $G \in U_{e}$ contains $s^{-1}s$.
Thus for every $G \in U_{e}$ we have that $sGs^{-1} \subseteq G$.
Hence the set $Z = V_{s} \cap \mathbf{d}^{-1}(V_{e})$ is an open subset of $\mbox{Iso}(\mathsf{G}(S))$.
It follows by Proposition~\ref{prop:one} that every ultrafilter in $Z$ is idempotent.
Thus $(sF^{\uparrow})^{\uparrow}$ is an idempotent ultrafilter.
We may therefore find $f \in F$ such that $f \leq s$.
But then $f \leq \phi (s) \leq s$.
Hence $\phi (s) \in F$.
But this contradicts the fact that $\sigma (s) \in F$.
It follows that $Y \cap U_{e} \neq \emptyset$.
We have therefore proved that $U_{\sigma (s)}$ is the closure of $Y$.
\end{proof}

The above proposition will be used mainly in the case where $s= g$ is a unit.
In this case, if $F \subseteq E(S)$ is an ultrafilter then so is $gFg^{-1}$.
We therefore have the following special case of Proposition~\ref{prop:tempest}.

\begin{proposition}\label{prop:labor} Let $S$ be a fundamental Boolean inverse $\wedge$-monoid.
For each unit $g$, we have that
$$U_{\sigma (g)} = \mathsf{cl}(\{ F \colon F \in \mathsf{X}(S) \mbox{ and } gFg^{-1} \neq F   \}).$$
\end{proposition}

Thus for units $g$, the idempotent $\sigma (g)$ translates into the support in the usual topological sense.

The next result is important in translating from topology to algebra.

\begin{lemma}\label{lem:pingo} Let $S$ be a fundamental Boolean inverse $\wedge$-monoid and let $g$ be a unit.
\begin{enumerate}

\item Let $F \subseteq E(S)$ be an ultrafilter such that $gFg^{-1} \neq F$.
Then there exists an idempotent $e \in F$ such that $e \perp geg^{-1}$.

\item  Let $F \subseteq E(S)$ be an ultrafilter such that $\sigma (g) \in F$.
For each $f \in F$ where $f \leq \sigma (g)$, 
there exists $0 \neq e \leq f$ such that $e \perp geg^{-1}$.

\end{enumerate}
\end{lemma} 
\begin{proof} (1) Since the ultrafilters $F$ and $gFg^{-1}$ are distinct and the structure space is Hausdorff,
there exist non-zero idempotents $e \perp f$ such that $F \in U_{e}$ and $gFg^{-1} \in U_{f}$.
Since $e \notin gFg^{-1}$, there exists $gig^{-1} \in gFg^{-1}$ such that $e (gig^{-1}) = 0$.
Put $j = ie$.
Then $j \in F$ and $j(gjg^{-1}) = 0$.

(2) Let $f \in F$ where $f \leq \sigma (s)$.
Then by Proposition~\ref{prop:labor}, 
the open set $U_{f}$ contains an element $G$ such that $gGg^{-1} \neq G$.
Thus by (1), there is an idempotent $e \in G$, which can also be chosen to satisfy $e \leq f$, such that $e(geg^{-1}) = 0$.
\end{proof}

\subsection{The principal case}

Recall that a groupoid $G$ is said to be {\em principal} if $g,h \in G$ such that $g^{-1}g = h^{-1}h$ and $gg^{-1} = hh^{-1}$ imply that $g = h$.
We now have our next main result.
Observe that since $\phi (a) \leq a$ we have that $\phi (a) \leq a^{-1}a,aa^{-1}$.

\begin{proposition}\label{prop:two} Let $S$ be a Boolean inverse $\wedge$-monoid.
Then $\mathsf{G}(S)$ is principal if, and only if, 
$$U_{\phi (s)} = \{ F \colon F \in U_{s^{-1}s} \mbox{ and } sFs^{-1}  \subseteq F   \},$$
for all $s \in S$.
\end{proposition}
\begin{proof} Suppose first that 
$$U_{\phi (s)} = \{ F \colon F \in U_{s^{-1}s} \mbox{ and }  sFs^{-1}  \subseteq F   \},$$
for all $s \in S$.
Let $A,B \in \mathsf{G}(S)$ such that 
$\mathbf{d}(A) = \mathbf{d}(B)$ 
and
$\mathbf{r}(A) = \mathbf{r}(B)$.
Let $\mathbf{d}(A) = \mathbf{d}(B) = F^{\uparrow}$ where $F \subseteq E(S)$ is an ultrafilter.
Then 
$A = (aF)^{\uparrow}$ for any $a \in A$,
and
$B = (bF)^{\uparrow}$ for any $b \in B$.
By assumption, $(aFa^{-1})^{\uparrow} = (bFb^{-1})^{\uparrow}$.
It is easy to check that $a^{-1}bF (a^{-1}b)^{-1} \subseteq F$,
and that $b^{-1}aa^{-1}b \in F$.
Thus, by assumption, $\phi (a^{-1}b)  \in F$.
It follows that $a^{-1}b \in F^{\uparrow}$.
Hence $aa^{-1}b \in A$ and so $b \in A$.
But then it is immediate that $A = B$.

To prove the converse, assume that $\mathsf{G}(S)$ is principal.
It is enough to prove that 
$$ \{ F \colon F \in U_{s^{-1}s} \mbox{ and } sFs^{-1}  \subseteq F   \}
\subseteq
U_{\phi (s)}.$$ 
Let $sFs^{-1} \subseteq F$.
Put $A = (sF)^{\uparrow}$.
Then $\mathbf{d}(A) = F^{\uparrow}$ and $\mathbf{r}(A) = F^{\uparrow}$.
It follows that $A = F^{\uparrow}$.
Thus there is an idempotent $f \in F$ such that $f \leq s$.
But then $f \leq \phi (f)$ and so $\phi (f) \in F$, as required.
\end{proof}

We shall now translate the above result into an internal condition on $S$.
We adapt the idea contained in the first paragraphs of \cite[Section 7]{Kumjian}.

\begin{lemma}\label{lem:mirren} Let $S$ be a Boolean inverse $\wedge$-monoid.
Then $\mathsf{G}(S)$ is principal if, and only if, for each $s \in S$ we have that
$\sigma (s) = \bigvee_{i=1}^{m} e_{i}$ where $e_{i}se_{i} = 0$ for $1 \leq i \leq e$. 
\end{lemma}
\begin{proof} Suppose first that $\mathsf{G}(S)$ is principal.
Let $s \in S$.
Define 
$$L = \{e \leq s^{-1}s \colon 0 = ese \}.$$
We prove that
$$U_{\sigma (s)} = \bigcup_{e \in L} U_{e}.$$
Let $F \in U_{\sigma (s)}$.
Then $s^{-1}s \in F$ and so $sFs^{-1} \neq F$ by Proposition~\ref{prop:two}.
It follows that there exist $e,f \in F$ such that
$e (sfs^{-1}) = 0$.
Put $i = ef s^{-1}s $.
Then $i \in F$, $i \leq s^{-1}s$ and $i (sis^{-1}) = 0$.
It follows that $isi = 0$
since $i \leq e$ and $sis^{-1} \leq sfs^{-1}$.
We have shown that $i \in L$ and $F \in U_{i}$.
Thus the lefthandside is contained in the righthandside.
The reverse inclusion holds since if $e \in F$ where $e \in L$ and $\phi (s) \in F$
then $s \in F^{\uparrow}$ and so $0 = ese \in F^{\uparrow}$, which is a contradiction.
By compactness, we may assume thate $L$ is finite.
The result now follows.

To prove the converse, suppose that 
$\sigma (s) = \bigvee_{i=1}^{m} e_{i}$ where $e_{i}se_{i} = 0$ for $1 \leq i \leq e$. 
Let $F \in U_{s^{-1}s}$ such that $sFs^{-1} \subseteq F$.
Suppose that $\phi (s) \notin F$.
Then $\sigma (s) \in F$.
But $F$ is an ultrafilter and so a prime filter.
It follows that $e_{i} \in F$ for some $i$.
Then $se_{i}s^{-1} \in F$ and so $0 = e_{i}se_{i}s^{-1} \in F$, which is a contradiction.
Thus the result follows by Proposition~\ref{prop:two}.
\end{proof}

\begin{remark} {\em Boolean inverse $\wedge$-monoids that satisfy the above condition are the analogues of 
those inverse semigroups that act {\em relatively freely} \cite[Chapter 1, Proposition 2.13]{Renault}.}
\end{remark}

We shall restate the above lemma in a more striking form.

\begin{proposition}\label{prop:infinitesimal} Let $S$ be a Boolean inverse $\wedge$-monoid.
Then $\mathsf{G}(S)$ is principal if, and only if, for each $s \in S$ we have that $s = e \vee s_{1} \vee \ldots \vee s_{m}$
where $e$ is an idempotent, 
each $s_{i}$ is an infinitesimal for $1 \leq i \leq m$,
and $e \perp (s_{1} \vee \ldots \vee s_{m})$.
\end{proposition}
\begin{proof} Suppose first that $\mathsf{G}(S)$ is principal.
Then by Lemma~\ref{lem:mirren}, 
for each $s \in S$ we have that
$\sigma (s) = \bigvee_{i=1}^{m} e_{i}$ where $e_{i}se_{i} = 0$ for $1 \leq i \leq e$. 
By Lemma~\ref{lem:cooper}, we may write
$$s = \phi (s) \vee se_{1} \vee \ldots \vee se_{m}.$$
Now $e_{i}se_{i} = 0$  which gives $(se_{i})^{2} = 0$, and so the $se_{i}$ are infinitesimals.

We now prove the converse.
We assume that each element can be written in the stated form and deduce that $\mathsf{G}(S)$ is principal.
We use Lemma~\ref{lem:mirren}.
For each $s \in S$ we have that $s = e \vee s_{1} \vee \ldots \vee s_{m}$
where $e$ is an idempotent and each $s_{i}$ is an infinitesimal for $1 \leq i \leq m$. 
Let $f \leq s$ be any idempotent.
Then $f = f \wedge s$ and so $f = (f \wedge e) \vee (f \wedge s_{1}) \vee \ldots \vee (f \wedge s_{m})$
by Lemma~\ref{lem:meets_and_joins}.
It follows that each $f \wedge s_{i}$ is an idempotent less than an infinitesimal and so must be 0.
It follows that $f \leq e$.
We have proved that $e = \phi (s)$.
It now readily follows that $s\sigma (s) = s_{1} \vee \ldots \vee s_{m}$.
Thus
$\sigma (s) = s_{1}^{-1}s_{1} \vee \ldots \vee s_{m}^{-1}s_{m} \leq s^{-1}s$.
Put $e_{i} = s_{i}^{-1}s_{i}$.
Then $s_{i} = se_{i}$.
Then $se_{i}se_{i} = 0$ and so $e_{i}se_{i} = 0$, as required.
\end{proof}

\begin{remark}{\em To understand the significance of Proposition~\ref{prop:infinitesimal}, we begin with the following quote from \cite[page 3]{Renault}:
\begin{quote} 
``From our point of view, the most interesting groupoids are principal groupoids.
Their $C^{\ast}$-algebras appear as genuine generalizations of matrix algebras.''
\end{quote}
We may translate this into the language of Boolean inverse monoids.
Observe that in a finite symmetric inverse monoid $I(X)$ every element may be written as a finite orthogonal join
of an idempotent and infinitesimals.
To see why, note that partial bijections of the form $x \mapsto y$, where $x,y \in X$ and $x \neq y$, are infinitesimals,
and that the partial bijections of the form $x \mapsto x$ are idempotents.
Clearly, each partial bijection of $X$ can be written as an orthogonal union of an idempotent and infinitesimals.
This agrees with the above result because the associated groupoid is principal.
Thus Boolean inverse $\wedge$-monoids where $\mathsf{G}(S)$ is a principal groupoid 
may be regarded as direct generalizations of finite symmetric inverse monoids.
There is no terminology for Boolean inverse $\wedge$-monoids of this type,
so we shall refer to them as {\em principal} as well.
It can be verified directly, by showing that an element centralizes the idempotents only if it is itself an idempotent,
that every principal Boolean inverse $\wedge$-monoid is fundamental.}
\end{remark}

\subsection{The $0$-simplifying case}

Let $G$ be a groupoid with set of identities $G_{o}$.
We say that two identities $e$ and $f$ are {\em related} if there is an element $g \in G$ such that
$e \stackrel{g}{\rightarrow} f$.
This is an equivalence relation on the set $G_{o}$ and the equivalence classes are called {\em $G$-orbits}.
The $G$-orbit containing $e$ is denoted by $G(e)$.
A subset of $G_{o}$ is called an {\em invariant set} if it is a union of $G$-orbits.
Both $\emptyset$ and $G_{o}$ are invariant sets called the {\em trivial invariant sets}.
We say that two elements $g,h \in G$ are related if their identities $g^{-1}g$ and $h^{-1}h$ are related.
This relation is an equivalence relation.
Its equivalence classes are called {\em connected components}.
An {\em invariant subset} of $G$ is any union of connected components.
Lenz \cite{Lenz} remarks that the equivalence of (2) and (3) below is well-known in the \'etale case.

\begin{lemma}\label{lem:oinky} Let $G$ be an \'etale groupoid. 
Then the following are equivalent.
\begin{enumerate}

\item Every $G$-orbit is a dense subset of $G_{o}$.

\item Any non-empty invariant subset of $G_{o}$ is dense in $G_{o}$.

\item There are no non-trivial open invariant subsets of $G_{o}$.

\item There are no non-trivial open invariant subsets of $G$.

\end{enumerate}
\end{lemma}
\begin{proof}
(1)$\Rightarrow$(2). 
This is immediate since a non-empty invariant set is a union of $G$-orbits.

(2)$\Rightarrow$(3).  Let $U \subseteq G_{o}$ be a non-empty open invariant subset.
Suppose that $G_{o} \neq U$.
Let $e \in G_{o} \setminus U$.
Then $G(e) \subseteq  G_{o} \setminus U$ since if $G(e)$ had a non-empty intersection with $U$, 
it would be wholly contained within $U$.
However, by (2), since $G(e)$ is an invariant subset of $G_{o}$,
it is also a dense subset.
It follows that $U \cap G(e) \neq \emptyset$, which is a contradiction.

(3)$\Rightarrow$(1). Let $G(e)$ be the $G$-orbit containing $e \in G_{o}$.
Suppose that $G(e)$ is not dense in $G_{o}$.
Then there exists a non-empty open set $U \subseteq G_{o}$ such that $G(e) \cap U = \emptyset$.
Since $G$ is \'etale, the set $GU$ is open.
Thus $V = \mathbf{r}(GU)$ is open and contains $U$.
It is evidently an invariant subset.
Thus, by assumption, $V = G_{o}$.
Hence $G_{o} =  \mathbf{r}(GU)$. 
It follows that $G(e) \cap U \neq \emptyset$, which is a contradiction. 

(3)$\Rightarrow$(4). Let $U \subseteq G$ be a non-empty open invariant subset.
Then $U_{o} = U \cap G_{o}$ is a non-empty open invariant subset of $G_{o}$
with the property that $U = GU_{o}$.
By assumption, $U_{o} = G_{o}$.
Hence $U = G$, as required.

(4)$\Rightarrow$(3). Let $U \subseteq G_{o}$ be a non-empty open invariant subset of $G_{o}$.
Then $GU_{o}$ is a non-empty open invariant subset of $G$.
By assumption, $GU_{o} = G$.
Thus $U = G_{o}$, as required.
\end{proof}

We shall say that an \'etale groupoid is {\em minimal} if it satisfies any one of the equivalent conditions in Lemma~\ref{lem:oinky}. 
The following was sketched in \cite{Lawson12} and is a reformulation of a result in \cite{Lenz}.

\begin{proposition}\label{prop:sweep} Let $S$ be a Boolean inverse $\wedge$-monoid and let $\mathsf{G}(S)$ be its associated Boolean groupoid.
Then the groupoid $\mathsf{G}(S)$ is minimal if, and only if, $S$ is $0$-simplifying.
\end{proposition}
\begin{proof} We prove that there is an order isomorphism between the set of $\vee$-ideals of $S$ and the set of open invariant subsets of $\mathsf{G}(S)$.
The result will then follow by Lemma~\ref{lem:oinky}.

Let $I \subseteq S$ be a $\vee$-ideal.
Define 
$$\mathsf{O}(I) = \bigcup_{s \in I} V_{s}.$$
By construction, this is an open subset of $\mathsf{G}(S)$.
Let $A \in \mathsf{O}(I)$. 
Suppose that $B$ is an ultrafilter such that $B^{-1} \cdot B = A^{-1} \cdot A$.
We may write $B = (tA^{-1} \cdot A)^{\uparrow}$ for some $t$ where $t^{-1}t \in A^{-1} \cdot A$.
Observe that $ts^{-1}s \in B$.
But $ts^{-1}s \in I$.
It follows that $B \in \mathsf{O}(I)$.
Taken together with the dual result, we deduce that $\mathsf{O}(I)$ is an invariant subset.

Let $U \subseteq \mathsf{G}(S)$ be an open invariant subset.
Define
$$\mathsf{I}(U) = \{s \in S \colon V_{s} \subseteq U \}.$$  
Let $t \in S$ be any element and let $s \in \mathsf{I}(U)$.
We prove that $st \in \mathsf{I}(U)$.
Thus we need to show that $V_{st} \subseteq U$.
Let $A \in V_{st}$.
Then $st(st)^{-1} \in A \cdot A^{-1}$.
In particular, $ss^{-1} \in  A \cdot A^{-1}$.
Put $B = (A \cdot A^{-1} s)^{\uparrow} \in V_{s}$.
Thus $B \in U$.
But $A \cdot A^{-1} = B \cdot B^{-1}$ and so $B \in U$, since $U$ is an invariant subset.
It follows that $V_{st} \subseteq U$, as required.

It is now routine to check that the maps $I \mapsto \mathsf{O}(I)$ and $U \mapsto \mathsf{I}(U)$
are mutually inverse order-preserving maps.
\end{proof}

\subsection{The $0$-simple case}

Every $0$-simple Tarski monoid is a $0$-simplifying Tarski monoid.
We shall now investigate what the difference between these two classes is.

\begin{lemma}\label{lem:purely_infinite} Let $S$ be a $0$-simple Tarski monoid and let $e \in S$ be any non-zero idempotent.
Then we may find a pair of elements $x,y \in S$ such that $\mathbf{d}(x) = e = \mathbf{d}(y)$,
and $\mathbf{r}(x)$ and $\mathbf{r}(y)$ are orthogonal,
and $\mathbf{r}(x) \vee \mathbf{r}(y) \leq e$.
\end{lemma}
\begin{proof} Since $S$ is atomless, there is a non-zero idempotent $f < e$.
From the fact that $S$ is $0$-simple, we may find an element $x$ such that
$e \stackrel{x}{\rightarrow} e_{1} \leq f < e$.
Similarly, we may find an element $y$ such that
$e \stackrel{y}{\rightarrow} e_{2} \leq e\overline{f} < e$.
\end{proof}

A non-zero idempotent $e$ is said to be {\em properly infinite} if 
we may find a pair of elements $x,y \in S$ such that $\mathbf{d}(x) = e = \mathbf{d}(y)$,
and $\mathbf{r}(x)$ and $\mathbf{r}(y)$ are orthogonal,
and $\mathbf{r}(x) \vee \mathbf{r}(y) \leq e$.
An inverse monoid is said to be {\em purely infinite} if every non-zero idempotent is properly infinite.
This terminology is generalized from \cite{Matui13}.

\begin{remark} Let $e$ be a properly infinite idempotent in the inverse monoid $S$.
Then there is a monoid homomorphism $P_{2} \rightarrow eSe$, where $P_{2}$ is the polycyclic monoid on two generators.
This homomorphism is an embedding since $P_{2}$ is congruence-free.
\end{remark}

We may therefore rephrase Lemma~\ref{lem:purely_infinite} in the following terms.

\begin{corollary}\label{cor:beatrice} 
In a $0$-simple Tarski monoid every non-zero idempotent is properly infinite.
In particular, each local submonoid contains a copy of $P_{2}$.
\end{corollary}

This result will lead us to an exact formulation of the difference between $0$-simple and $0$-simplifying.

\begin{lemma}\label{lem:benedick} Let $S$ be a Tarski monoid and let $e$ and $f$ be any two non-zero idempotents such that $e \preceq f$. 
Then we may find elements 
$x_{1}, \ldots, x_{m}$ such that
$\mathbf{r}(x_{i}) \leq f$ for $1 \leq i \leq m$ 
and 
$e = \bigvee_{i=1}^{m}  \mathbf{d}(x_{i})$
where this is an orthogonal join of idempotents.
\end{lemma}
\begin{proof} From the definition of $\preceq$ we may find such elements $y_{j}$ such that the following hold
$y_{1}, \ldots, y_{m}$ such that
$\mathbf{r}(y_{i}) \leq f$ for $1 \leq i \leq m$ 
and 
$e = \bigvee_{i=1}^{m}  \mathbf{d}(y_{i})$.
Put $e_{i} = \mathbf{d}(y_{i})$.
Define idempotents $f_{1}, \ldots, f_{n}$ as follows
$f_{1} = e_{1}$, $f_{2} = e_{2} \overline{e_{1}}$, \ldots, $f_{n} = e_{n} \overline{(e_{1} \vee \ldots \vee e_{n-1})}$.
These idempotents are pairwise orthogonal and their join  is $e$.
Observe that $f_{i} \leq e_{i}$.
Define $x_{i} = y_{i}f_{f}$.
Then $\mathbf{d}(x_{i}) = f_{i}$.
Clearly $\mathbf{r}(x_{i}) \leq f$.
\end{proof}

We now have the following result suggested by \cite[Proposition 4.11]{Matui13}

\begin{proposition}\label{prop:leonato} Let $S$ be a Tarski monoid.
Then the following are equivalent.
\begin{enumerate}

\item  $S$ is $0$-simple.

\item $S$ is $0$-simplifying and purely infinite.

\end{enumerate}
\end{proposition}
\begin{proof} (1)$\Rightarrow$(2).
Every $0$-simple semigroup is $0$-simplifying,
and we proved in Corollary~\ref{cor:beatrice}, that in a $0$-simple Tarski monoid every non-zero idempotent is properly infinite.

(2)$\Rightarrow$(1).
Let $e$ and $f$ be any two non-zero idempotents.
From the fact that the monoid is $0$-simplifying, and  Lemma~\ref{lem:benedick}, 
we may find elements $w_{1}, \ldots, w_{n}$ such that
$e = \bigvee_{i=1}^{n} \mathbf{d}(w_{i})$ is an orthogonal join
and $\mathbf{r}(w_{i}) \leq f$.
From the fact that the monoid is purely infinite, we may find elements $a$ and $b$ such that
$\mathbf{d}(a) = f = \mathbf{d}(b)$ and $\mathbf{r}(a),\mathbf{r}(b) \leq f$
and $\mathbf{r}(a)$ and $\mathbf{r}(b)$ are orthogonal.
Thus, in particular, $a^{-1}b = 0$ and $a^{-1}a = e = b^{-1}b$.
Define the elements $v_{1}, \ldots, v_{n}$ as follows:
$v_{1} = a$, $v_{2} = ba$, $v_{3} = b^{2}a$, \ldots, $v_{n} = b^{n-1}a$.
Observe that $\mathbf{d}(v_{i}) = f$ and that the $\mathbf{r}(v_{i}) \leq f$ are pairwise orthogonal.
Consider now the elements $v_{1}w_{1}, \ldots, v_{n}w_{n}$.
It is easy to check that these elements are pairwise orthogonal.
We may therefore form the join $w = \bigvee_{i=1}^{n} v_{i} w_{i}$.
Clearly, $\mathbf{d}(w) = e$ and $\mathbf{r}(w) \leq f$.
The result now follows by Lemma~\ref{lem:helen}.
\end{proof}

The following generalizes part (3) of \cite[Proposition 4.10]{Matui13}.

\begin{lemma}\label{lem:conrade} Let $S$ be a $0$-simple Tarski monoid.
Let $e$ and $f$ be idempotents such that $e \neq 1$ and $f \neq 0$.
Then there is an invertible element $g$ such that $geg^{-1} \leq f$.
\end{lemma}
\begin{proof}
Suppose first that $f \bar{e} \neq 0$.
Since $S$ is $0$-simple, there exists $a \in S$ such that $\mathbf{d}(a) = e$ and $\mathbf{r}(a) \leq f \bar{e}$.
Clearly, $\mathbf{d}(a)$ and $\mathbf{r}(a)$ are orthogonal.
Thus $a^{2} = 0$.
By Lemma~\ref{lem:spooks}, we may  define $g = a \vee a^{-1} \vee i$, a unit, where $i = 1 \overline{(\mathbf{d}(a) \vee \mathbf{r}(a))}$.
Thus $ie = 0$.
We have that $geg^{-1} \leq f$.

Suppose now that $f \bar{e} = 0$.
Then $f < e$.
By the above result, we may find an invertible element $u$ such that
$ueu^{-1} \leq \bar{e}$.
Similarly, we may find an invertible element $v$ such that $v\bar{e}v^{-1} \leq f$.
Thus $vu e (vu)^{-1} \leq f$, as required.
\end{proof}

\section{Proof of Theorem~\ref{them:one}}

The following summarizes what we may deduce using Theorem~\ref{them:duality} when applied to Boolean inverse $\wedge$-monoids
under the conditions of  Proposition~\ref{prop:blop}, Proposition~\ref{prop:one} and Proposition~\ref{prop:sweep}.

\begin{theorem}\label{them:qi} Under non-commutative Stone duality,
$0$-simplifying, fundamental Tarski monoids correspond
to minimal, essentially principal, second countable Hausdorff  \'etale groupoids with unit space the Cantor space.
Under this correspondence, the group of units of the inverse monoid corresponds to the group of compact-open bisections of the groupoid.
\end{theorem}

The above theorem, combined with \cite[Theorem 3.10]{Matui13}, provides the indirect proof of our Theorem~\ref{them:one}.
However, we wish to give a direct proof, which is the goal of this section.

\subsection{The axioms}

Recall that in the case where $g$ is a unit, we have that $\sigma (g) = \overline{\phi (g)}$.

\begin{lemma}\label{lem:midge} Let $S$ be a Boolean inverse $\wedge$-monoid, and let $g$ and $h$ be units.
\begin{enumerate}

\item $\sigma (g) = 0$ if, and only if, $g = 1$.

\item $\sigma (g^{-1}) = \sigma (g)$.

\item $\sigma (gh) \leq \sigma (g) \vee \sigma (h)$.

\end{enumerate}
\end{lemma}
\begin{proof} (1) One direction follows since $\phi (1) = 1$.
To prove the converse, suppose that $\phi (g) = 1$.
Then $1 \leq g$ and so $g = 1$.

(2) For any idempotent $e$ we have that $e \leq g$ if and only if $e \leq g^{-1}$.
It follows that $\phi (g) = \phi (g^{-1})$.

(3) From $\phi (g) \leq g$ and $\phi (h) \leq h$ we get that $\phi (g)\phi (h) \leq gh$.
Thus $\phi (g)\phi (h) \leq \phi (gh)$.
The result now follows.
\end{proof}

We now state three axioms  that will play a crucial role in proving the spatial realization theorem.
They are translations into our language of those given in \cite[Definition 3.1]{Matui13}.
See also \cite{Fremlin}.

\begin{itemize}

\item[{\rm (F1)}] {\em Enough involutions. }For each ultrafilter $F \subseteq E(S)$ and each $e \in F$, there is a non-trivial involution $t$ such that $\sigma (t) \in F$ and $\sigma (t) \leq e$.

\item[{\rm (F2)}] {\em Shrinking.} For each non-trivial involution $t$ and non-zero idempotent $e \leq \sigma (t)$ there exists a non-trivial unit $g$
such that 
$$\sigma (g) \leq e \vee tet = (te)^{-1}te \vee te(te)^{-1}$$ 
and $gFg^{-1} = tFt$ for all $F \in U_{\sigma (g)}$.

\item[{\rm (F3)}] {\em Enough non-involutions.} For each non-zero idempotent $e$, there exists a non-involution unit $g$ such that $\sigma (g) \leq e$.

\end{itemize}

\begin{proposition}\label{prop:hope} 
Let $S$ be a $0$-simplifying, fundamental Tarski monoid.
Then the axioms (F1), (F2) and (F3) hold.
\end{proposition}
\begin{proof} 

(F1) holds.
Let $F \subseteq E(S)$ be an ultrafilter and $e \in F$.
Then by Proposition~\ref{prop:george}, there is an infinitesimal $a$ such that $a \in eSe$ and $a^{-1}a \in F$.
By Lemma~\ref{lem:spooks}, the element 
$$t = a \vee a^{-1} \vee f$$ 
is an involution, where $f = \overline{a^{-1}a} \,  \overline{aa^{-1}}$.
From $f \leq t$ we get that $f \leq \phi (t)$ and so $\sigma (t) \leq \bar{f} \leq a^{-1}a \vee aa^{-1} \leq e$.
Suppose that $\phi (t) \in F$.
Then by Lemma~\ref{lem:meets_and_joins}, exactly one of $a \wedge 1$ or $a^{-1} \wedge 1$ or $f$ is in $F$.
But $a^{-1}a \in F$ and we quickly get contradictions in all three cases.
It follows that $\sigma (t) \in F$, as required.

(F2) holds.
Let $t$ be a non-trivial involution and let $0 \neq e \leq \sigma (t)$.
By Proposition~\ref{prop:labor}, we may find an ultrafilter $F \subseteq E(S)$ containing $e$ such that $tFt \neq F$.
By Lemma~\ref{lem:pingo},  we may therefore find a non-zero idempotent $f \in F$ such that $f \leq e$ and $f \perp tft$.
By Proposition~\ref{prop:george}, there is therefore an infinitesimal $a$ such that $a^{-1}a \vee aa^{-1} \leq f$ and $a^{-1}a \in F$.
It follows that $a^{-1}a$, $aa^{-1}$, $ta^{-1}at$ and $taa^{-1}t$ are mutually orthogonal. 
Define 
$$i = \overline{(a^{-1}a)} \, \overline{(ta^{-1}at)}.$$
Observe that $ta^{-1}a \perp a^{-1}at$
and that $ta^{-1}a \vee a^{-1}at$ has the same domain as codomain, which is $a^{-1}a \vee ta^{-1}at$.
Thus by Lemma~\ref{lem:spooks}
$$g = ta^{-1}a \vee a^{-1}at \vee i$$
is a unit.
From $i \leq g$ we have that $i \leq \phi (g)$ and so 
$$\sigma (g) \leq a^{-1}a \vee ta^{-1}at \leq e \vee tet.$$
Let $G \in U_{\sigma (g)}$.
We shall show that $gGg^{-1} = tGt$.
Observe that $a^{-1}a \vee ta^{-1}at \in G$.
Let $j \in G$.
We calculate $gjg^{-1}$.
This can be written as a join.
This join contains $ta^{-1}ajt$.
Thus if $a^{-1}a \in G$ then $gjg^{-1} \in tFt$.
If $a^{-1}a \notin G$ then $ta^{-1}at \in G$.
The join representing $gjg^{-1}$ contains $a^{-1}atjt$.
This can be written as $t (ta^{-1}at)jt$, which is again an element of $tGt$.
It follows that $gGg^{-1} \subseteq tGt$ and, 
since both are ultrafilters, 
it follows that they are equal.

(F3) holds. Let $e$ be a non-zero idempotent.
By Lemma~\ref{lem:horsa}, we may find infinitesimals $a,b \in eSe$ such that $ab$ is a restricted product and $ab$ is an infinitesimal.
Put 
$$g = a \vee a^{-1} \vee i,$$ 
an involution, where $i = \overline{a^{-1}a}\, \overline{aa^{-1}}$.
Put 
$$h = b \vee b^{-1} \vee j,$$ 
an involution, 
where $j = \overline{b^{-1}b} \, \overline{bb^{-1}}$.
We claim that $gh$ is a unit of order 3.
Observe that $ij \leq gh$ and so $ij \leq \phi(gh)$.
It follows that 
$\sigma (gh) \leq (a^{-1}a \vee aa^{-1})(b^{-1}b \vee bb^{-1}) \leq e$.
We shall prove that $gh$ has order 3.
We write the identity as an orthogonal join of four idempotents
$$1 = a^{-1}a = bb^{-1}, \,
2 = aa^{-1}, \,
3 = b^{-1}b, \,
4 = j\overline{(aa^{-1})}.$$
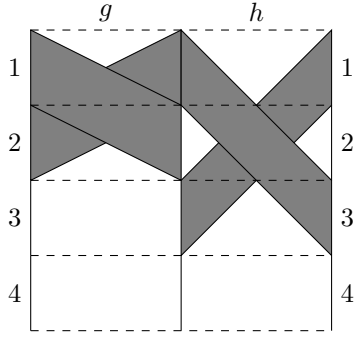
\begin{center}
\begin{tikzpicture}[xscale=2]

\draw [fill=gray] (0,2) -- (1,3) -- (1,4) -- (0,3) -- (0,2);
\draw [fill=gray] (0,3) -- (1,2) -- (1,3) -- (0,4) -- (0,3);
\draw [fill=gray] (1,1) -- (2,3) -- (2,4) -- (1,2) -- (1,1);
\draw [fill=gray] (2,1) -- (2,2) -- (1,4) -- (1,3) -- (2,1);

\draw [dashed] (0,4) -- (1,4) -- (2,4);
\draw [dashed] (0,3) -- (1,3) -- (2,3);
\draw [dashed] (0,2) -- (1,2) -- (2,2);
\draw [dashed] (0,1) -- (1,1) -- (2,1);
\draw [dashed] (0,0) -- (1,0) -- (2,0);

\draw (0,0) -- (0,1);
\draw (0,1) -- (0,2);
\draw (2,2) -- (2,3);
\draw (1,0) -- (1,1);
\draw (2,0) -- (2,1);

\node [above] at (0.5,4) {$g$};
\node [above] at (1.5,4) {$h$};

\node [right] at (2,3.5) {$1$};
\node [right] at (2,2.5) {$2$};
\node [right] at (2,1.5) {$3$};
\node [right] at (2,0.5) {$4$};

\node [left] at (0,3.5) {$1$};
\node [left] at (0,2.5) {$2$};
\node [left] at (0,1.5) {$3$};
\node [left] at (0,0.5) {$4$};


\end{tikzpicture}
\end{center}
We look at what $gh$ does to these four idempotents.
Schematically we get the permutation $(132)$.
It is clear that $gh \neq 1$ and $(gh)^{2} \neq 1$ and that $(gh)^{3} = 1$.
In detail,
$gh = ab \vee a^{-1}j \vee  ib^{-1}  \vee ij$.
Using basic Boolean algebra properties, notably $ef = 0 \Rightarrow e \leq \bar{f}$, we get the following
$$
b^{-1}b \stackrel{ab}{\longrightarrow} aa^{-1},
\quad
aa^{-1} \stackrel{a^{-1}j}{\longrightarrow} a^{-1}a,
\quad
bb^{-1} \stackrel{ib^{-1}}{\longrightarrow} b^{-1}b.$$
\end{proof}

\subsection{Local groups}

If $e \in E(S)$, define
$$U(e) = \{g \in U(S) \colon \sigma (g) \leq e \}.$$
In the lemma below, 
part (3) is the inverse monoid version of \cite[Lemma 3.2]{Matui13}.

\begin{lemma}\label{lem:tom} Let $S$ be a Boolean inverse $\wedge$-monoid.
\begin{enumerate}

\item $U(e)$ is a subgroup of $U(S)$.

\item If $e \leq f$ then $U(e) \subseteq U(f)$.

\item If (F1) or (F3) hold, then  $e \leq f$ if, and only if, $U(e) \subseteq U(f)$.

\end{enumerate}
\end{lemma}
\begin{proof} (1) This follows by Lemma~\ref{lem:midge}.

(2) Immediate.

(3)  Only one direction needs proving.
Let $U(e) \subseteq U(f)$.
We use the fact that in a Boolean algebra, we have that $e \leq f$ if and only if $e \bar{f} = 0$.
Suppose that $e \bar{f} \neq 0$.
By (F3), there exists a non-trivial unit $t$ such that $\sigma (t) \leq e \bar{f}$.
Clearly, $t \in U(e)$.
If $t \in U(f)$ then $\sigma (t) \leq f$.
Thus $\sigma (t) = 0$ which implies that $t = 1$, a contradiction since $t$ is non-trivial.
It follows that $t \notin U(f)$, which is a contradiction.

\end{proof}

We call $U(e)$ the {\em local subgroup at $e$}.

\begin{remark}{\em 
In what follows, we shall use the following important consequence of Lemma~\ref{lem:bingo}.
Let $g$ and $h$ be units in a fundamental Boolean inverse $\wedge$-monoid.
Then $g = h$ if, and only if, $gFg^{-1} = hFh^{-1}$ for all ultrafilters $F \subseteq E(S)$.
We call this process {\em testing against ultrafilters}.}
\end{remark}

\begin{lemma}\label{lem:lisboa} Let $S$ be a fundamental Boolean inverse $\wedge$-monoid.
Suppose that the unit $g$ is the identity when tested against the ultrafilters in $U_{e}$.
Then $e \leq \phi (g)$.
\end{lemma}
\begin{proof} Suppose that $e \nleq \phi (g)$.
Then there exists an ultrafilter $F$ such that $e \in F$ and $\phi (g) \notin F$.
Thus $\sigma (g) \in F$.
By Proposition \ref{prop:labor}, there exists an ultrafilter $G \in U_{e}$ such that $gGg^{-1} \neq G$,
which is a contradiction.
\end{proof}

Let $t$ be a fixed involution.
Denote by $C_{t}$ the centralizer of $t$ in $U(S)$.
Define $Z_{t} \leq C_{t}$ as follows
$$Z_{t} = \{ s \in C_{t} \colon s^{2} = 1, \, (\forall a \in C_{t}) [s,asa^{-1}] = 1 \}.$$
Define
$$S_{t} = \{ a^{2} \colon a \in U(S), (\forall s \in Z_{t}) [a,s] = 1\}.$$
Define
$$W_{t} = \{a \in U(S) \colon  (\forall b \in S_{t})    [a,b] = 1 \}.$$
Clearly, $t \in C_{t}$, and $t \in Z_{t}$.

In the following lemmas, the involution $t$ is fixed.

\begin{lemma}\label{lem:booboo} Let $S$ be a fundamental Boolean inverse $\wedge$-monoid satisfying (F3).
If $s \in Z_{t}$ and $F \in U_{\phi (t)}$ then $sFs = F$.
\end{lemma}
\begin{proof} We prove the following.
If $s$ is an involution such that there exists $F \in U_{\phi (t)}$ where $sFs \neq F$ then there exists $a \in C_{t}$ such that $[s,asa^{-1}] \neq 1$.
In other words, $s \notin Z_{t}$.

By Lemma~\ref{lem:pingo}, from the fact that $\phi (t) \in F$ and $sFs \neq F$, we may find $e \in F$ such that $e \leq \phi (t)$ and $e(ses) = 0$.
In particular, $e \sigma (t) = 0$.
By (F3), we may find a unit $a$ such that $a^{2} \neq 1$ and $\sigma (a) \leq e$.
Thus, in particular, $\sigma (a) \sigma (t) = 0$.
Therefore by Lemma~\ref{lem:jerry}, we have that $at = ta$, and so we have shown that $a \in C_{t}$.

By Lemma~\ref{lem:midge},  we have that $\sigma (a^{2}) \leq \sigma (a)$.
By Proposition~\ref{prop:labor},  there is an ultrafilter $G \subseteq E(S)$ such that $\sigma (a^{2}) \in G$ such that $a^{2}Ga^{-2} \neq G$.
It is easy to check that as a consequence $G$, $aGa^{-1}$ and $a^{-1}Ga$ are distinct ultrafilters in $E(S)$.

We shall now prove that $s$ and $asa^{-1}$ do not commute.

Since $e(ses) = 0$ and $\sigma (a) \leq e$, we have that $\sigma (a)s \sigma (a)s = 0$ 
By Lemma~\ref{lem:jerry}, we have that $\sigma (a)\sigma (sas) = 0$.

We calculate  $sasa^{-1}s G asa^{-1}s$.
Observe that since $\sigma (a) \in G$,
we have that $\sigma (sas) \in sa^{-1}Gas$.
It follows that $\sigma (a) \notin sa^{-1}Gas$ and so $\phi (a) \in sa^{-1}Gas$.
Thus $a(sa^{-1}Gas)a^{-1} = sa^{-1}Gas$.
We therefore get that 
$$sasa^{-1}s G sasa^{-1} = a^{-1}Ga.$$

We now calculate $asa^{-1}sGsasa^{-1}$.
Observe that $\sigma (a) = \sigma (a^{-1})$.
Now $\sigma (sa^{-1}s) \in sGs$.
Thus $\sigma (a^{-1}) \notin sGs$ and so $\phi (a^{-1}) \in sGs$.
It follows that $a^{-1}sGsa = sGs$.
It now readily follows that  
$$asa^{-1}sGsasa^{-1} = aGa^{-1}.$$

Thus $s$ and $asa^{-1}$ do not commute.
\end{proof}

\begin{lemma}\label{lem:matui_two}
Let $S$ be a fundamental Boolean inverse $\wedge$-monoid satisfying (F1), (F2) and (F3).
Let $0 \neq e \leq \sigma (t)$.
Then there exists $s \in Z_{t}$ such that 
$\sigma (s) \leq e \vee te(te)^{-1}$ and $sFs = tFt$ for all $F \in U_{\sigma (s)}$.
\end{lemma}
\begin{proof} By (F2),  we may find an element $s$ satisfying all the given conditions except possibly $s \in Z_{t}$.
We therefore need to verify three conditions:
$s^{2} = 1$, $[s,t] = 1$ and $s$ commutes with every element of the form $asa^{-1}$ where $a \in C_{t}$.

To show that $s^{2} = 1$, it is enough to let it act on those ultrafilters in $U_{\sigma (s)}$.
We calculate $s^{2}Fs^{-2}$.
We have that $sFs^{-1} = tFt$.
Now $\sigma (s) \in F$ implies that $s \sigma (s)s^{-1} \in sFs^{-1}$.
But $s \sigma (s)s^{-1} = \sigma (s)$.
Thus $\sigma (s) \in tFt$.
It follows that $s(tFt)s^{-1} = t^{2}Ft^{2} = F$.
We have therefore proved that $s^{2} = 1$ and so $s^{-1} = s$.

We next prove that $[s,t] = 1$.
To do this, we shall prove that $stst = 1$ by testing it against all ultrafilters in $E(S)$.
There are two cases to check.
First let $F \in U_{\sigma (s)}$.
We calculate $ststFtsts$.
We have that $tFt = sFs$.
Thus 
$$sts(tFt)sts
=
sts(sFs)sts
=
s(tFt)s
=
s(sFs)s
=
F$$
since $s^{2} = 1$.
Now let $F \in U_{\phi (s)}$.
Suppose that $\phi (s) \notin tFt$.
Then $\sigma (s) \in tFt$.
Thus $s(tFt)s = ttFtt = F$.
Hence $s\sigma (s)s \in F$ and so $\sigma (s) \in F$, which is a contradiction.
It follows that $\phi (s) \in tFt$.
Thus $sts(tFt)sts = st (tFt) ts = sFs = F$.
We have therefore proved that in all cases $stst$ acts as the identity and so is the identity.

Finally, let $a \in C_{t}$.
We shall prove that $s$ commutes with $asa^{-1}$.
There are two cases to check.
First let $F \in U_{\sigma (s)}$.
Then
$$asa^{-1} sFs asa^{-1} 
= asa^{-1} tFt asa^{-1}
= t asa^{-1} F asa^{-1} t.$$
Now $\sigma (s) \in F$ implies that $\sigma (asa^{-1}) \vee \sigma (s) \in asa^{-1}Fasa^{-1}$.
There are now two cases.
First, suppose that $\sigma (s) \in asa^{-1}Fasa^{-1}$.
Then
$ t asa^{-1} F asa^{-1} t = s asa^{-1} F asa^{-1} s$
and in this case commutativity holds.
Second,  suppose that
$\sigma (asa^{-1}) \in  asa^{-1}Fasa^{-1}$.
Then $\sigma (s) \in a^{-1}Fa$.
We have that
$$asa^{-1} sFs asa^{-1} 
= asa^{-1} tFt asa^{-1}
= t as(a^{-1} F a)sa^{-1} t
= F$$
and
$$s as(a^{-1} F a)sa^{-1} s 
=
s at(a^{-1} F a)ta^{-1} s 
=
tsFst
= F.$$
Now let $F \in U_{\phi (s)}$.
Then
$$asa^{-1} s F s asa^{-1} = asa^{-1} F asa^{-1}.$$
We now calculate $s asa^{-1}F asa^{-1} s$.
If $\phi (s) \in asa^{-1} F asa^{-1}$ then the result follows.
We may therefore assume that $\sigma (s) \in asa^{-1} F asa^{-1}$.
It follows that $\sigma (asa^{-1} s asa^{-1}) \in F$.
But  $\sigma (asa^{-1} s asa^{-1}) \leq \sigma (asa^{-1}) \vee \sigma (s)$.
It follows that $\sigma (asa^{-1})$ or $\sigma (s)$ belongs to $F$.
But $\phi (s) \in F$ and so $\sigma (s) \notin F$.
Therefore $\sigma (asa^{-1}) \in F$.
Then $\sigma (s) \in a^{-1}Fa$.
It follows that
$$asa^{-1}Fasa^{-1} = ata^{-1}Fata^{-1} = tFt.$$
Hence
$$asa^{-1} s F s asa^{-1} = asa^{-1} F asa^{-1} = tFt,$$
and
$$s asa^{-1}F asa^{-1} s 
= st F ts = tFt$$
where the final equality is obtained from the fact that $\phi (s) \in F$ and so $\phi (s) \in sFs$.
\end{proof}

\begin{lemma}\label{lem:matui_three} 
Let $S$ be a fundamental Boolean inverse $\wedge$-monoid satisfying (F1), (F2) and (F3).
Let $e$ be any non-zero idempotent such that $e \sigma (t) = 0$.
Then there exists an element $a$ such that $\sigma (a) \leq e$ and $a^{2} \in S_{t}$ is non-trivial.
Observe that $\sigma (a^{2}) \leq \sigma (a) \leq e$.
\end{lemma}
\begin{proof} By (F3), there exists a unit $a$ such that $a^{2} \neq 1$ and $\sigma (a) \leq e$.
From Lemma~\ref{lem:jerry} and $\sigma (a) \sigma (t) = 0$ we have that $[a,t] = 1$.
Let $s \in Z_{t}$.
We prove that $[a,s] = 1$.
By Lemma~\ref{lem:booboo}, we have that $s \in Z_{t}$ and $F \in U_{\phi (t)}$ implies that $sFs = F$.
We prove that $sa = as$ by testing these two elements against ultrafilters.

Suppose first that $\sigma (t) \notin F$.
Then $\phi (t) \in F$.
It follows that $asFsa^{-1} = aFa^{-1}$.
Now $\phi (t) \in F$ implies that $a\phi (t)a^{-1} \in  aFa^{-1}$.
But since $a$ and $t$ commute, we have that $\phi (t) = a \phi (t) a^{-1}$.
Thus $\phi (t) \in aFa^{-1}$.
It follows that $saFa^{-1}s = aFa^{-1}$.

Now suppose that $\sigma (t) \in F$.
Then $\phi (a) \in F$ and so $saFa^{-1}s = sFs$.
Now $\sigma (t) \in F$ implies that $s \sigma (t) s \in sFs$.
But $s$ and $t$ commute and so $s \sigma (t)s = \sigma (t)$.
Thus $\sigma (t) \in sFs$.
Then $\phi (a) \in sFs$ and so $a(sFs)a^{-1} = sFs$.
\end{proof}

\begin{lemma}\label{lem:matui_four}
Let $S$ be a fundamental Boolean inverse $\wedge$-monoid satisfying (F1), (F2) and (F3).
If $b \in S_{t}$ and $F \in U_{\sigma (t)}$ then $bFb^{-1} = F$.
\end{lemma}
\begin{proof} We shall prove the following claim:
let $a$ be a unit that commutes with every element in $Z_{t}$,
then if $tFt \neq F$ then either $aFa^{-1} = F$ or $aFa^{-1} = tFt$.

We show first why this claim proves the lemma.
Let $b \in S_{t}$.
Then $b = a^{2}$, where $a$ commutes with every element of $Z_{t}$.
Let $F \in U_{\sigma (t)}$.
First suppose that $tFt \neq F$.
If $aFa^{-1} = F$ then $bFb^{-1} = F$.
If $aFa^{-1} = tFt$ then $bFb = atFta^{-1} = taFa^{-1}F = ttFtt = F$
where we use the fact that $a$ commutes with $t$.
Now suppose that $tFt = F$.
We shall prove that $a^{2}Fa^{-2} = F$.
Suppose not.
Then $a^{2}Fa^{-2} \neq F$.
By Lemma~\ref{lem:pingo}, there exists an idempotent $e \in F$ such that $e \perp a^{2}ea^{-2}$ and $e \leq \sigma (t)$.
By Proposition~\ref{prop:labor}, there is an element $G \in U_{e}$ such that $tGt \neq G$.
It follows that $a^{2}Ga^{-2} = G$.
But then $e a^{2}ea^{-2} \neq 0$, which is a contradiction.

We now prove the claim.
Let $F$ be an ultrafilter in $E(S)$ containing $\sigma (t)$.
Suppose that to the contrary, $aFa^{-1}$, $F$ and $tFt$ are distinct.
We may therefore find a non-zero idempotent $e \leq \sigma (t)$ such that $e \in F$ are
$e$, $tet$ and $aea^{-1}$ are mutually orthogonal.
We now use Lemma~\ref{lem:matui_two}.
There exists $s \in Z_{t}$ such that $\sigma (s) \leq e \vee tet$ and $sGs = tGt$ for all $G \in U_{\sigma (s)}$.

We prove that $a$ and $s$ cannot commute which is a contradiction.
We claim that $saGa^{-1}s = aGa^{-1}$.
Accepting this, from $\sigma (s) \in G$, we get that $asGsa^{-1} = atGta^{-1}$.
Thus we need to prove that $aGa^{-1} \neq atGta^{-1}$.
From $\sigma (s) \in G$, we get that $e \vee tet \in G$.
But $e \perp tet$ thus $e \in G$ or $tet \in G$ but not both.
Suppose that $e \in G$.
Then $aea^{-1} \in aGa^{-1}$ and $ateta^{-1} \in atGa^{-1}t$.
But $e \perp tet$ implies that $aea^{-1} \perp ateta^{-1}$ and so $aGa^{-1} \neq atGta^{-1}$.
Suppose that $tet \in G$.
Then $ateta^{-1} \in aGa^{-1}$ and $aea^{-1} \in atGa^{-1}t$.
Thus once again $aGa^{-1} \neq atGta^{-1}$.

It therefore only remains to prove the claim that  $saGa^{-1}s = aGa^{-1}$.
Either $e \in G$ or $tet \in G$ but not both.
Suppose that $e \in G$.
Then $aea^{-1} \in aGa^{-1}$.
But $aea^{-1}$ is othogonal to $e \vee tet$.
Thus $\sigma (s) \notin aGa^{-1}$ and so $\phi (s) \in aGa^{-1}$.
Hence $saGa^{-1}s = aGa^{-1}$.

We shall now prove that $tet \notin G$.
Suppose to the contrary that $tet \in G$.
Then $ateta^{-1}  \in aGa^{-1}$ and so $t(aea^{-1})t \in aGa^{-1}$.
Since $a$ and $s$ commute, we have that $\sigma (s) \in aGa^{-1}$.
It follows that $e \vee tet \in aGa^{-1}$.
Thus $tet \vee e \in t(aGa^{-1})t$.
But $t(aea^{-1})t$ and $tet \vee t$ are orthogonal, which is a contradiction.
\end{proof}

The key result of this section is the following.

\begin{proposition}\label{prop:matui_five}
Let $S$ be a fundamental Boolean inverse $\wedge$-monoid satisfying (F1), (F2) and (F3).
Then
$$W_{t} = U(\sigma (t)).$$
\end{proposition}
\begin{proof} We prove first that $U(\sigma (t)) \subseteq W_{t}$.
Let $g \in U(\sigma (t))$.
Then $\sigma (g) \leq \sigma (t)$.
We need to prove that $g$ commutes with every element of $S_{t}$.
By Lemma~\ref{lem:matui_four}, $b \in S_{t}$, and $\sigma (t) \in F$ implies that $bFb^{-1} = F$.
Observe that since $b$ acts as the identity on $U_{\sigma (t)}$, it must map $U_{\phi (t)}$ to itself.
Suppose first that $F \in U_{\sigma (t)}$.
Then
$gbFb^{-1}g^{-1} = gFg^{-1}$.
Since $\sigma (g) \leq \sigma (t)$, we have that $\phi (t) \leq \phi (g)$.
Thus $g$ is the identity on $U_{\phi (t)}$ and therefore maps $U_{\sigma (t)}$ to itself.
It follows that $\sigma (t) \in gFg^{-1}$ and so 
$bgFg^{-1}b^{-1} = gFg^{-1}$.
Now suppose that $F \in U_{\phi (t)}$.
Then $gbFb^{-1}g^{-1} = bFb^{-1}$
and
$bgFg^{-1}b^{-1} = bFb^{-1}$,
where we have used our observations above.

We now prove that $W_{t} \subseteq U(\sigma (t))$.
Let $a \in W_{t}$ and $\sigma (t) \notin F$.
Suppose that $aFa^{-1} \neq F$.
Then there is a non-zero idempotent $e \in F$ such that $e \sigma (t) = 0$ and $e(aea^{-1}) = 0$.
By Lemma~\ref{lem:matui_three}, there exists a non-trivial $b \in S_{t}$ such that $\sigma (b) \leq e$.
We claim that $b \neq aba^{-1}$.
To see why, observe that $\sigma (b) a \sigma (b) a^{-1} = 0$ and so $\sigma (b) \sigma (aba^{-1}) = 0$,
from which it immediately follows that $b \neq aba^{-1}$.
We may now apply Lemma \ref{lem:lisboa} to deduce that $\sigma (a) \leq \sigma (t)$.
\end{proof}

\subsection{Constructing the isomorphism}

In the previous section, we saw that local subgroups of the form $U(\sigma (t))$, where $t$ is an involution, can be characterized algebraically.
This is the main result needed in what follows.

\begin{lemma}\label{lem:splott} Let $S_{1}$ and $S_{2}$ be Tarski monoids satisfying conditions (F1), (F2) and (F3).
Let $\theta \colon U(S_{1}) \rightarrow U(S_{2})$ be an isomorphism of unit groups.
Let $s,t \in U(S_{1})$ be involutions.
Then
\begin{enumerate}

\item $\sigma (t) \leq \sigma (s)$ if, and only if, $\sigma (\theta (t)) \leq \sigma (\theta (s))$.

\item $\sigma (t) \perp \sigma (s)$ if, and only if, $\sigma (\theta (t)) \perp \sigma (\theta (s))$.

\end{enumerate}
\end{lemma}
\begin{proof} (1) Suppose that $\sigma (t) \leq \sigma (s)$.
Then by Lemma~\ref{lem:tom}, we have that $U(\sigma (t)) \leq U(\sigma (s))$.
By Proposition~\ref{prop:matui_five}, we have that $W_{t} \leq W_{s}$.
We now work entirely within the groups of units.
We get $\theta (W_{t}) \leq \theta (W_{s})$ and so $W_{\theta (t)} \leq W_{\theta (s)}$.
By Proposition~\ref{prop:matui_five}, we get that $U(\sigma (\theta (t))) \leq U(\sigma (\theta (s)))$.
Thus by Lemma~\ref{lem:tom}, we get that $\sigma (\theta (t)) \leq \sigma (\theta (s))$.
The reverse implication follows since $\theta^{-1}$ is an isomorphism, and the result now follows.

(2) Suppose that $\sigma (t) \sigma (s) \neq 0$.
By (F1), there exists a non-trivial involution $r$ such that $\sigma (r) \leq \sigma (t)\sigma (s)$.
By Lemma~\ref{lem:tom}, we have that
$U(\sigma (r)) \leq U(\sigma (t)) \cap U(\sigma (s))$.
By part (1), we have that $U(\sigma (\theta (r))) \leq U(\sigma (\theta(t))) \cap U(\sigma (\theta(s)))$,
and so by Lemma~\ref{lem:tom}, we have that $\sigma (\theta (r)) \leq \sigma (\theta (t)) \sigma (\theta (s))$.
In particular, $\sigma (\theta (t)) \sigma (\theta (s)) \neq 0$.
\end{proof}

If $F \subseteq E(S)$ is an ultrafilter, we denote by $T(F)$ the set of all involutions $t$ such that $\sigma (t) \in F$.

\begin{lemma}\label{lem:susan} Let $S$ be a Tarski monoid satisfying conditions (F1), (F2) and (F3).
\begin{enumerate}

\item  Let $F \subseteq E(S)$ be an ultrafilter.
Define
$$F^{\sigma} = \{ \sigma (t) \colon t \in T(F) \},$$
the {\em support skeleton} of $F$.
Then $F^{\sigma}$ is a filter base and $(F^{\sigma})^{\uparrow}$.

\item Every open set in the structure space $\mathsf{X}(S)$ is a finite union of open sets of the form $U_{\sigma (t)}$
where $t$ is an involution.

\end{enumerate}
\end{lemma}
\begin{proof} (1)  Let $e_{1}, e_{2} \in F^{\sigma}$.
Then there exists $e \in F^{\sigma}$ such that $e \leq e_{1}, e_{2}$.
To see why, observe that the product $e_{1} e_{2}$ is non-zero since the idempotents belong to the ultrafilter $F$.
Thus by condition (F1), there exists a non-zero involution $t$ such that $\sigma (t)  \in F$ and $\sigma (t) \leq e_{1} e_{2}$.
We may therefore put $e = \sigma (t)$.
It follows that $F^{\sigma}$ is a  filter base.
In fact, $(F^{\sigma})^{\uparrow} = F$.
Since if $e \in F$ then by condition (F1) there exists a non-trivial involution $t$ 
such that $\sigma (t) \in F$ and $\sigma (t) \leq e$.
By construction $\sigma (t) \in F^{\sigma}$.
We have therefore shown that every ultrafilter in $E(S)$ is determined uniquely by its support skeleton.

(2) Let $U_{e}$ be an open set where $e \neq 0$.
Let $F \in U_{e}$.
Then by condition (F1), there is an involution $t$ such that $\sigma (t) \in F$ and $\sigma (t) \leq e$.
It follows that $F \in U_{\sigma (t)}$.
Put $\Lambda$ equal to the set of all such $t$ as the $F$ vary over $U_{e}$.
By construction $U_{e} \subseteq \bigcup_{t \in \Lambda} U_{\sigma (t)}$ and in fact equality holds.
We now use compactness to get thet $U_{e} = \bigcup_{i=1}^{m} U_{\sigma (t_{i})}$ where $t_{i}^{2} = t_{i}$.
We are also able to deduce that $e = \bigvee_{i=1}^{m} \sigma (t_{i})$.
\end{proof}

\begin{proposition}\label{them:matui_main} Let $S_{1}$ and $S_{2}$ be Tarski monoids satisfying conditions (F1), (F2) and (F3).
Let $\alpha \colon U(S_{1}) \rightarrow U(S_{2})$ be an isomorphism.
Then there exists a homeomorphism $\beta \colon \mathsf{X}(S_{1}) \rightarrow \mathsf{X}(S_{2})$ of structure spaces such that the following two conditions hold.
\begin{enumerate}

\item For every $F \in \mathsf{X}(S_{1})$ and $g \in U(S_{1})$ we have that $\beta (gFg^{-1}) = \alpha (g) \beta (F) \alpha (g)^{-1}$.

\item If $g \in U(S_{1})$ and $F \in \mathsf{X}(S_{1})$ then $\sigma (g) \notin F$ if, and only if, $\sigma (\alpha (g)) \notin \beta (F)$.

\end{enumerate}
\end{proposition}
\begin{proof} We apply the lessons of Lemma~\ref{lem:susan} throughout.
We begin by constructing the homeomorphism $\beta$.
Let $F \in \mathsf{X}(S_{1})$.
Put 
$$Y = \{ \sigma (\alpha (t)) \colon t \in T(F) \}.$$
By Lemma~\ref{lem:splott} and the fact that $F^{\sigma}$ is a filter base,
it follows that $Y$ is a filter base.
It follows that $Y^{\uparrow}$ is a filter in $E(S_{2})$
and so $Y$ must be contained in at least one ultrafilter of $E(S_{2})$.
Let $Y \subseteq G,H$ where $G,H \subseteq E(S_{2})$ are distinct ultrafilters.
It follows that we may find non-zero idempotents $e \perp f$ such that
$G \in U_{e}$ and $H \in U_{f}$.
By condition (F1), we may find a non-trivial involution $a \in U(S_{2})$
such that $\sigma (a) \leq e$ and $\sigma (a) \in G$.
Since $\sigma (a)f = 0$, it follows that $\sigma (a) \notin H$.
We now consider the idempotent $\sigma (\alpha^{-1}(a))$.
Suppose first that $\sigma (\alpha^{-1}(a)) \in F$.
Then $\sigma (a) \in Y$, $Y \subseteq H$ and we showed above that $\sigma (a) \notin H$.
This is a contradiction.
It follows that $\sigma (\alpha^{-1}(a)) \notin F$.
It follows by (F1), that there is a non-trivial involution $b$ such that $\sigma (b) \in F$
and $\sigma (b) \sigma (\alpha^{-1}(a)) = 0$.
But by Lemma~\ref{lem:splott}, we have that
$\sigma (b) \perp \sigma (\alpha^{-1}(a))$ implies that $\sigma (\alpha (b)) \perp \sigma (a)$.
Now $\sigma (\alpha (b)) \in Y$ and $\sigma (a) \in G$.
This also gives us a contradiction.
It follows that the set $Y$ is contained in a {\em unique} ultrafilter of $E(S_{2})$.
We denote this unique ultrafilter by $\beta (F)$.
It is immediate from the above construction, 
that for each involution $t$, we have that $\beta (U_{\sigma (t)}) \subseteq U_{\sigma (\alpha (t))}$.
It is also evident from the above construction by symmetry that $\beta$ is bijective.

We now verify conditions (1) and (2).
It is immediate from our construction that (2) holds.
Let $g \in U(S_{1})$ and $F \in \mathsf{X}(S_{1})$.
Let $\sigma (t) \in F$ be an element of $F^{\sigma}$.
Then $\sigma (gtg^{-1}) \in gFg^{-1}$ and we observe that $gtg^{-1}$ is also an involution.
Thus $\sigma (\alpha (gtg^{-1})) \in \beta (gFg^{-1})$.
Therefore $\alpha (g) \sigma (\alpha (t)) \alpha (g)^{-1} \in \beta (gFg^{-1})$.
It is now straightforward to see that $\alpha (g) \beta (F) \alpha (g)^{-1} \subseteq \beta (gFg^{-1})$.
But equality now follows since both are ultrafilters.
\end{proof}

\begin{proposition}\label{prop:aber} Let $S$ and $T$ be two piecewise factorizable Boolean inverse $\wedge$-monoids.
Let $\alpha \colon U(S) \rightarrow U(T)$ be an isomorphism of groups
and let $\beta \colon \mathsf{X}(S) \rightarrow \mathsf{X}(T)$ be a homeomorphism of structure spaces.
We suppose in addition that the following two conditions hold.
\begin{enumerate}

\item For every $F \in \mathsf{X}(S)$ and $g \in U(S)$ we have that $\beta (gFg^{-1}) = \alpha (g) \beta (F) \alpha (g)^{-1}$.

\item If $g \in U(S)$ and $F \in X(S)$ then $\sigma (g) \notin F$ if, and only if, $\sigma (\alpha (g)) \notin \beta (F)$.

\end{enumerate}
Then $S$ and $T$ are isomorphic.
\end{proposition}
\begin{proof} We shall actually construct a bijective functor $\theta \colon \mathsf{G}(S) \rightarrow \mathsf{G}(T)$ which is also a homeomorphism.
The result will then follow.

Let $A \in \mathsf{G}(S)$.
Then $A = (a \mathbf{d}(A))^{\uparrow}$ where $a \in A$.
Since $S$ is piecewise factorizable, there is a unit $g \in A$ and so we may write
$A = (g \mathbf{d}(A))^{\uparrow}$.
Now $E(\mathbf{d}(A)) \subseteq E(S)$ is an ultrafilter.
We may therefore define
$$\theta (A) = (\alpha (g) \beta (E(\mathbf{d}(A)))^{\uparrow})^{\uparrow},$$
it only remaining to show that this is well-defined.
Let $g,h \in A$ both units.
Then $g^{-1}h \in \mathbf{d}(A)$.
This implies that $\phi (g^{-1}h) \in E(\mathbf{d}(A))$.
It follows that $\sigma (g^{-1}h) \notin E(\mathbf{d}(A))$.
We now use condition (2), to deduce that $\sigma (\alpha (g^{-1}h)) \notin \beta (E(\mathbf{d}(A)))$.
It follows that $\alpha (g^{-1}h) \in \beta (E(\mathbf{d}(A)))^{\uparrow}$.
Thus $\theta$ is well-defined.
If $A$ is an idempotent ultrafilter, it contains the identity.
It follows that $\theta$ maps identities to identities.
It is immediate that $\theta$ is surjective.
We prove that $\theta$ is injective.
Suppose that $\theta (A) = \theta (B)$.
Let $A = (gF^{\uparrow})^{\uparrow}$ and $B = (hG^{\uparrow})^{\uparrow}$ where $F,G \subseteq E(S)$ are ultrafilters and $g,h \in U(S)$.
By definition $\theta (A) = (\alpha (g) \beta (F)^{\uparrow})^{\uparrow}$ and $\theta (B) = (\alpha (g) \beta (G)^{\uparrow})^{\uparrow}$.
Since $\mathbf{d}(\theta (A)) = \mathbf{d}(\theta (B))$,
we have that $\beta (F)^{\uparrow} = \beta (G)^{\uparrow}$.
Therefore $E(\beta (F)^{\uparrow}) = E( \beta (G)^{\uparrow})$.
Hence $\beta (F) = \beta (G)$ and so $F = G$.
We also have that $\alpha (g^{-1}h) \in \beta (E(\mathbf{d}(A))^{\uparrow})^{\uparrow}$.
We now use condition (2) to deduce that $A = B$.
We have therefore defined a bijection that preserves identities.

We now prove that $\theta$ is a functor.
Suppose that $A \cdot B$ is defined.
By condition (1), we check that $\mathbf{d}(A) = \mathbf{r}(B)$ implies that $\mathbf{d}(\theta (A)) = \mathbf{r}(\theta (B))$.
It follows that $\theta (A) \cdot \theta (B)$ is defined.
We now use condition (1) again to show that $\theta (A \cdot B) =\theta (A) \cdot \theta (B)$.

It remains to show that $\theta$ is a homeomorphism.
Let $t \leq h$, a unit, where $t \in U(T)$.
Put $g = \alpha^{-1} (h)$
and
$\beta^{-1}(U_{t^{-1}t}) = U_{e}$.
Define $s = ge$.
We shall show that $\theta^{-1}(V_{t}) = V_{s}$.

Now let $s \leq g$, a unit, where $s \in S$.
Put $h = \alpha (g)$
and
$\theta (U_{s^{-1}s}) = U_{f}$.
Define $t = hf$.
We shall show that $\theta (V_{s}) = V_{t}$.
\end{proof}

\begin{center}
{\bf Proof of Theorem~\ref{them:one}}
\end{center}

Let $S$ and $T$ be two countably infinite Boolean inverse $\wedge$-monoids which are $0$-simplifying and fundamental.
Let $\alpha \colon U(S) \rightarrow U(T)$ be an isomorphism of unit groups.
By Corollary \ref{cor:judy}, both $S$ and $T$ are Tarski monoids.
By Proposition \ref{prop:dory}, they are both piecewise factorizable.
By Proposition \ref{prop:hope}, the axioms (F1), (F2) and (F3) all hold.
The isomorphism of $S$ and $T$ now follows by Proposition \ref{them:matui_main} and Proposition \ref{prop:aber}.


\end{document}